\documentclass[pdflatex]{amsart}
\usepackage{latexsym,amssymb,amsfonts,mathrsfs}
\usepackage{amsthm}
\usepackage{amsmath}
\usepackage[foot]{amsaddr}
\usepackage{enumerate}
\usepackage{comment}
\usepackage[latin1, utf8]{inputenc}
\usepackage{graphicx}
\usepackage{lmodern,xfrac,xcolor}
\usepackage[noadjust]{cite}

\usepackage{mathtools,enumerate}
\usepackage[title]{appendix}
\usepackage{scalerel}
\usepackage{relsize}
\usepackage{tikz}
\usepackage{mathtools,accents,hyperref}
\mathtoolsset{showonlyrefs}

\newtheorem{theorem}{Theorem}
\newtheorem{lemma}[theorem]{Lemma}
\newtheorem{proposition}[theorem]{Proposition}
\newtheorem{corollary}[theorem]{Corollary}

\theoremstyle{definition}
\newtheorem{definition}[theorem]{Definition}

\newtheorem{assumptions}[theorem]{Assumptions}
\newtheorem{example}[theorem]{Example}

\theoremstyle{remark}

\theoremstyle{remark}
\newtheorem{remark}[theorem]{Remark}

\newcommand{\e}{\mathrm{e}}
\newcommand{\ve}{\varepsilon}
\newcommand{\la}{\lambda}
\newcommand{\me}{\mathsf{e}}
\newcommand{\mv}{\mathsf{v}}

\newcommand{\mE}{\mathsf{E}}
\newcommand{\mV}{\mathsf{V}}
\newcommand{\mG}{\mathsf{G}}

\newcommand{\Ker}{\mathrm{Ker}}

\newcommand{\diag}{\mathrm{diag}}
\newcommand{\sgrT}{\mathbb{T}}

\makeatletter
\newcommand{\smallbullet}{} 
\DeclareRobustCommand\smallbullet{%
  \mathord{\mathpalette\smallbullet@{0.5}}%
}
\newcommand{\smallbullet@}[2]{%
  \vcenter{\hbox{\scalebox{#2}{$\m@th#1\bullet$}}}%
}

\makeatother

\newcommand{\comp}{\mathbb{C}}
\newcommand{\real}{\mathbb{R}}
\newcommand{\nat}{\mathbb{N}}

\newcommand{\Dir}{\mathbb{D}}

\newcommand{\mcD}{\mathcal{D}}

\newcommand{\mcH}{\mathcal{H}}

\newcommand{\mcY}{\mathcal{Y}}

\newcommand{\dx}{\, dx}
\newcommand{\dt}{\, dt}
\newcommand{\ds}{\, ds}

\DeclareMathOperator{\Tr}{Tr}
\DeclareMathOperator{\HS}{HS}

\numberwithin{equation}{section} \numberwithin{theorem}{section}

\begin{document}

\title{On the strong Feller property of the heat equation on quantum graphs with Kirchhoff noise}

\author{Mohamed Fkirine$^1$} \email{mohamed.fkirine@tuni.fi}
\address{$^1$Mathematics Research Centre, Faculty of Information Technology and Communication SciencesTampere University, Korkeakoulunkatu 7, Tampere, 33720, Finland}

\author{Mihály Kovács$^2$}\email{mkovacs@math.bme.hu}
\address{$^2$Faculty of Information Technology and Bionics, P\'azm\'any P\'eter Catholic University, Práter u. 50/A, Budapest, 1083, Hungary}


\author{Eszter Sikolya$^3$}\email{eszter.sikolya@ttk.elte.hu}
\address{$^3$Department of Applied Analysis and Computational Mathematics, Eötvös Loránd University, Pázmány P. stny. 1/c, Budapest, 1117, Hungary}

\begin{abstract}We consider a so-called quantum graph with standard continuity and Kirchhoff vertex conditions where the Kirchhoff vertex condition is perturbed by Gaussian noise. We show that the quantum graph setting is very different from the classical one dimensional boundary noise setting, where the transition semigroup is known to be strong Feller, by giving examples and counterexamples to the strong Feller property. In particular, when the graph is a tree, and there is noise present in all of the boundary vertices except one, then the transition semigroup associated with the problem is strong Feller at any time $T>0$. This turns out to be also a necessary condition for equilateral star graphs. We also comment on the existence and uniqueness of the invariant measure and the regularity of the solution.\end{abstract}



\maketitle

	\section{Introduction}\label{sec:intro}

	We consider a so-called quantum graph; that is, is a metric graph $\mG$, equipped with a diffusion operator on each edge and certain vertex conditions. Our terminology follows \cite[Chap.~1]{BeKu} (see also \cite{Be17} and \cite{Mu14}), we list here only the most important concepts. The graph $\mG$ consists of a finite set of vertices $\mV = \{\mv\}$ and a finite set $\mE = \{\me\}$ of edges connecting the vertices. We denote by $n=|\mV|$ the number of vertices. In general, a metric graph is assumed to have directed edges; that is edges having an origin and a boundary vertex. In our case, dealing with self-adjoint operators, we can just consider undirected edges. Each edge is assigned a positive length $\ell_{\me}\in (0,+\infty)$, and we denote by $x\in [0,\ell_{\me}]$ a coordinate of $\mG$. 
	
	The metric graph structure enables one to speak about functions $z$ on $\mG$, defined along the edges such that for any coordinate $x$, the function takes its value $z(x)$. If we emphasize that $x$ is taken from the edge $\me$, we write $z_{\me}(x)$. Thus, a function $z$ on $\mG$ can be regarded as a vector of functions that are defined on the edges, therefore we will also write
	\[z=\left(z_{\me}\right)_{\me\in \mE},\]
	and consider it as an element of a product function space. 
	
	To write down the vertex conditions in the form of equations, for a given function $z$ on $\mG$ and for each $\mv\in\mV$, we introduce the following notation. For any $\mv\in\mV$, we denote by $\mE_{\mv}$ the set of edges incident to the vertex $\mv$, and by $d_{\mv}=|\mE_{\mv}|$ the degree of $\mv$. Let $z_{\me}(\mv)$ denote the value of $z$ in $\mv$ along the edge $\me$ in the case $\me\in \mE_{\mv}$. Let $\mE_{\mv}=\{\me_1,\dots ,\me_{d_{\mv}}\}$, and define
	\begin{equation}\label{eq:Fv}
		Z(\mv)=\left(z_{\me}(\mv)\right)_{\me\in \mE_{\mv}}=\begin{pmatrix}
			z_{\me_1}(\mv)\\
			\vdots\\
			z_{\me_{d_{\mv}}}(\mv)
		\end{pmatrix}\in \real^{d_{\mv}},
	\end{equation}
	the vector of the function values in the vertex $\mv$. 
	
	For vertices with $d_{\mv}>1$, let $I_{\mv}$ be the bi-diagonal matrix
	\begin{equation}\label{eq:Iv}
		I_{\mv}=\begin{pmatrix}
			1 & -1 & & \\
			& \ddots & \ddots & \\
			& & 1 & -1 
		\end{pmatrix}\in\real^{(d_{\mv}-1)\times d_{\mv}}.
	\end{equation}
	It is easy to see that if we set
	\begin{equation}\label{eq:contv}
		I_{\mv}Z(\mv)=0_{\real^{d_{\mv}-1}},
	\end{equation} 
	this means that all the function values coincide in $\mv$.  If it is satisfied for each vertex $\mv\in\mV$ with $d_{\mv}>1$ for a function $z$ on $\mG$, which is continuous on each edge, including the one-sided continuity at the endpoints, then we call $z$ \emph{continuous on} $\mG$.\smallskip
	
	Similarly, for a function $z$ on $\mG$ which is differentiable on each edge; that is, $z'_{\me}$ exists for each $\me\in\mE$ including the one-sided derivatives at the endpoints, we set
	\begin{equation}\label{eq:Fvv}
		Z'(\mv)=\left(z'_{\me}(\mv)\right)_{\me\in\mE_{\mv}}=\begin{pmatrix}
			z'_{\me_1}(\mv)\\
			\vdots\\
			z'_{\me_{d_{\mv}}}(\mv)
		\end{pmatrix}\in \real^{d_{\mv}},
	\end{equation}
	the vector of the function derivatives in the vertex $\mv$. We will assume throughout the paper that derivatives are taken in the directions away from the vertex $\mv$ (i.e.~into the edge), see \cite[Sec.~1.4.]{BeKu}.\\
 
Let $(\Omega,\mathscr{F},\mathbb{P})$ is a complete probability space endowed with a right-continuous filtration $\mathbb{F}=(\mathscr{F}_t)_{t\in [0,T]}$.
	Let the process 
	\[(\beta(t))_{t\in [0,T]}=\left(\left(\beta_{\mv}(t)\right)_{t\in [0,T]}\right)_{\mv\in\mV},\] 
	be an $\real^{n}$-valued Brownian motion (Wiener process) with covariance matrix 
	\[Q\in \real^{n\times n}\]
	with respect to the filtration $\mathbb{F}$; that is,
	$(\beta(t))_{t\in [0,T]}$ is $(\mathscr{F}_t)_{t\in [0,T]}$-adapted and for all $t>s$, $\beta(t)-\beta(s)$ is independent of $\mathscr{F}_s$.
 
	We consider the problem written formally as
	\begin{equation}\label{eq:stochnet}
		\left\{\begin{aligned}
			\dot{z}_{\me}(t,x) & =  (c_{\me} z_{\me}')'(t,x)-p_{\me}(x) z_{\me}(t,x), &x\in (0,\ell_{\me}),\; t\in(0,T],\; \me\in\mE,\;\; & (a)\\
			0 & = I_{\mv}Z(t,\mv),\;   &t\in(0,T],\; \mv\in\mV,\;\;& (b)\\
			\dot{\beta}_{\mv}(t) & = C(\mv)^{\top}Z'(t,\mv), & t\in(0,T],\; \mv\in\mV,\;\;& (c)\\
			z_{\me}(0,x) & =  z_{0,\me}(x), &x\in [0,\ell_{\me}].\; \me\in\mE.\;\;& (d)
		\end{aligned}
		\right.
	\end{equation}
Here $\dot{z}_{\me}$ and $z_{\me}'$ denote the time and space derivative, respectively, of $z_{\me}$. For each $\mv\in\mV$, $Z(t,\mv)$ and $Z'(t,\mv)$ denote the vector of the function values in $\mv$ introduced in \eqref{eq:Fv} and \eqref{eq:Fvv}, respectively, for the function on $\mG$ defined as
	\[z(t,\cdot)=\left(z_{\me}(t,\cdot)\right)_{\me\in\mE}.\]
	Analogously to \eqref{eq:Fv}, for each $\mv\in\mV$ and $\mE_{\mv}=\{\me_1,\dots ,\me_{d_{\mv}}\}$ the 1-row matrix in $(\ref{eq:stochnet}c)$ is defined by
	\[C(\mv)^{\top}=\left(c_{\me_1}(\mv), \dots ,c_{\me_{d_{\mv}}}(\mv)\right)\in\real^{1\times d_{\mv}}.\]
	Equations $(\ref{eq:stochnet}b)$ assume continuity of the function $z(t,\cdot)$ on the metric graph $\mG$, cf.~\eqref{eq:contv}. If we set $0\in\real^{d_{\mv}}$ on the left-hand-side of equations $(\ref{eq:stochnet}c)$, we obtain the usual \textit{Kirchhoff vertex conditions} for the system which is the generalization of the Neumann boundary condition on an interval, cf.~(\ref{eq:netcp}c). Therefore, we refer to the noise in equation $(\ref{eq:stochnet}c)$ as \textit{Kirchhoff noise}.
	
	It is clear by definition, that $(\ref{eq:stochnet}b)$ consists of 
	\begin{equation}\label{eq:conteqnumber}
		\sum_{\mv\in\mV}(d_{\mv}-1)=2|\mE|-|\mV|=2|\mE|-n
	\end{equation} 
	equations. At the same time, $(\ref{eq:stochnet}c)$ consists of $n$ equations. Hence, we have altogether $2|\mE| $ (boundary or vertex) conditions in the vertices. 
	
 The functions $c_{\me}$ are (variable) diffusion coefficients or conductances, and we assume that
	\[c_{\me}\in \mathrm{Lip}[0,\ell_{\me}],\quad 0< c_0\leq c_{\me}(x),\quad\text{for all }x\in [0,\ell_{\me}]\text{ and }  \me\in\mE.\]
	
	The functions $p_{\me}$ are non-negative, bounded functions, hence
	
	\begin{equation}
		0\leq p_{\me}\in L^{\infty}(0,\ell_{\me}),\quad \me\in\mE.
	\end{equation}
	In equation~$(\ref{eq:stochnet}d)$ we pose the initial conditions on the edges.\\

 In the recent paper \cite{KS23} the authors investigated (a nonlinear version of) \eqref{eq:stochnet} and proved, among other things, that the transition semigroup associated with the \eqref{eq:stochnet} is Feller. However, the strong Feller property of the solution has not been established there and this property is in the focus of the present paper. The model  \eqref{eq:stochnet} can be viewed as the quantum graph analogue of the model considered by Da Prato and Zabczyk in \cite{DPZ93} in case of a single interval and classical Neumann boundary noise. In the latter paper the authors show that the transition semigroup associated with the problem is always strong Feller. It turns out that in our setting this is not true in general. We show that in certain situations the transition semigroup is strong Feller, however, we also show that there are ample situations where the strong Feller property fails to hold even for simple configurations such as a equilateral star graph if the system is not noisy enough.\\

One general strategy for investigating the strong Feller property, see for example, \cite[Chapter 9.4]{DPZbook} is to consider an associated null controllability problem. While there exist several results for control problems on quantum graphs, these mainly deal with the wave equation, see for example, \cite{Avd1,AvdMikh08,AvdZhao22,Bel,LLS94,Zuazuabook} for an admittedly incomplete list. The Dirichlet boundary vertex null controllability of the heat equation for metric trees is mainly approached via an exact controllability problem for the wave equation \cite{Avd1,AvdMikh08} with the exception of the recent paper \cite{Barcena}, where the the Hilbert Uniqueness Method is used, in
particular, the authors study the observability of the adjoint system with a Carleman inequality. We are not aware of results concerning the Neumann vertex control of the heat equation on metric graphs. There are several possible approaches to the problem. One could study the adjoint problem as in \cite{Barcena} and use the Hilbert Uniqeness Method with a Carleman inequality, consider the exact controllability of the associated wave equation or try to use the method of moments directly. This latter approach would be rather problematic to use due to the lack of sufficient information on the separation of consecutive eigenvalues of the system operator, in general. We will hence consider the exact Neumann-controllability of the associated wave equation which was investigated in \cite{AvdZhao22} through a method originally introduced in \cite{AvdZhao21} for Dirichlet boundary control on metric trees. In \cite{AvdZhao22}, unlike in \cite[Chapter 5]{LLS94}, it is not assumed that there is a boundary vertex with zero Dirichlet boundary condition. This latter is essential in \cite[Chapter 5]{LLS94} as it ensures that the energy in fact is a norm on the state space. We therefore rely on the theory of so-called ST-active sets from \cite{AvdZhao22}, which is essential in dealing with control problems for the wave equation for general metric graphs, to show that the wave equation is exactly controllable on metric trees if one applies Neumann controls at all of the boundary vertices except for possibly one of them, where  Neumann zero boundary condition is assumed. One can then apply the theory of vector-exponentials in a similar fashion as in \cite{Avd1,AvdMikh08} to obtain the null-controllability of the heat equation on metric trees if one applies Neumann controls at all of the boundary vertices except one, where  Neumann zero boundary condition is assumed. \\

Having this null-controllability result at hand, Theorem \ref{thm:main} shows  that if the covariance matrix $Q=\diag(q_{\mv})_{\mv\in\mV}$ in  \eqref{eq:kovmtx} is diagonal, and $q_{\mv}\neq 0$ for all boundary vertices except for, possibly, one of them, then the transition semigroup associated with the problem is strong Feller at any time $T>0$. It turns out that for the Laplace operator on the equilateral star-graph the conditions of Theorem \ref{thm:main} is also necessary. Finally, we mention here that for the wave equation it is known that for  
general metric trees (at least for Dirichlet control) it is also necessary to apply controls at all of the boundary vertices except for, possibly, one of them. One, unfortunately, cannot use this result to conclude the same general result for the null controllability of the heat equation as one only has the exact controllability of the wave equation yields the null controllability of the heat equation on appropriate state spaces and not vice versa.
\\
	
     The paper is organized as follows.
     In Section \ref{sec:control} we give a short introduction into control theory for operator semigroups containing the notions and results we need in subsequent sections. \\
     
    In Section \ref{sec:strongFeller} we consider a general abstract boundary noise problem \eqref{eq:opAKmax1gen} and identify the associated abstract null-control problem to consider in order to prove the strong Feller property of the transition semigroup. \\
    
    In Section \ref{sec:strongFellernetwork} we first show that the problem \eqref{eq:stochnetKnoise} can be treated using the setting of the Section \ref{sec:strongFeller}. In Theorem \ref{thm:main} we show that if $c_e=1$ and the covariance matrix $Q=\diag(q_{\mv})_{\mv\in\mV}$ in  \eqref{eq:kovmtx} is diagonal, and $q_{\mv}\neq 0$ for all boundary vertices except for, possibly one of them, then the transition semigroup associated with the problem is strong Feller at any time $T>0$.
    Then, in Example \ref{ex:eqst}, we demonstrate that for equilateral star graphs with zero potential and the Laplace operator on the edges the transition semigroup is strong Feller if and only if noise is present at all boundary vertices, except for, possibly, one of them; that is, the assumptions of Theorem \ref{thm:main} are also necessary in this case. We end the section with Example \ref{ex:loop} which shows that the existence of a loop in the graph always destroys the strong Feller property. We would like to highlight that these examples show that quantum graph setting is very different from the classical one dimensional boundary noise setting, where the transition semigroup is known to be strong Feller~\cite{DPZ93}. \\
    
    In Section \ref{sec:invarmeasure}, in Theorem \ref{inv_measure}, we give a necessary and sufficient condition for the existence of the invariant measure for the solutions of the general problem \eqref{eq:opAKmax1gen}. The result \cite[Thm.~9.1.1 (iii)]{DPZbook96} states that if the transition semigroup is strong Feller, then the invariant measure exists if and only if the semigroup $\mathbb T$ is exponentially stable. We show in Corollary \ref{cor:inv} that for the quantum graph problem \eqref{eq:stochnetKnoise} the assumption being strong Feller can be omitted.\\
    
    In Appendix \ref{sec:regularity} we prove a regularity result in Proposition \ref{prop:solH14} which shows that the solutions of \eqref{eq:stochnetKnoise} have a continuous version in the Sobolev space of the edges of order less than $\frac{1}{4}$ without the uniform boundeness assumption of the vertex values of eigenfunctions of \cite[Theorem 3.8]{KS23}.\\
    
    In Appendix \ref{sec:STtree} we cite definitions and results from \cite{AvdZhao21, AvdZhao22} and show, in Proposition \ref{prop:STtree}, how they can be applied for the situation in Theorem \ref{thm:main}.

 \section{A short introduction to control theory}\label{sec:control}
 
Throughout this section, we assume that $H$, $U$ are Hilbert spaces and $T>0$ is a real number. First, we introduce some classical background on semigroup theory. Let $A$ be the generator of a strongly continuous semigroup $\mathbb{T} = (\mathbb{T}(t))_{t\geq0}$ on $H$, with growth bound $\omega_0(\mathbb{T})$. We denote by $H_{-1}$ the completion of $H$ with respect to the norm $\|x\|_{-1} = \|R(\la, A)x\|$, where $R(\la, A) := (\la - A)^{-1}$, and $\la$ is an arbitrary (but fixed) element in the resolvent set $\rho(A)$. It is easy to verify that $H_{-1}$ is independent of the choice of $\la$, since different values of $\la$ lead to equivalent norms on $H_{-1}$. The space $H_{-1}$ is isomorphic to $\mcD(A^*)^*$  and we have $\mcD(A) \subset H \subset H_{-1}$ densely and with continuous embedding. The semigroup $\mathbb{T}$ extends to a strongly continuous semigroup on $H_{-1}$, whose generator is an extension of $A$, with domain $H$. We denote the extensions of $\mathbb{T}$ and $A$ by $\mathbb{T}_{-1}:= (\mathbb T_{-1}(t))_{t\geq0}$ and $A_{-1}$, respectively.\\

 Next, consider the control system
 	\begin{equation}\label{eq:conpro1}
		\begin{cases}
			\dot{z}(t)  =A_{-1}z(t)+Bu(t), & t\in(0,T],\\
			z(0) =z_0\in H,&
		\end{cases}
	\end{equation}
 where $B\in\mathcal{L}(U,H_{-1})$ and $u\in L^2(0,T;U)$ is a control function. The mild solution of the equation \eqref{eq:conpro1} is given by
	\begin{equation}\label{eq:solcontrol}
		z(t)=\mathbb{T}(t)z_0+\Phi_tu, \qquad u\in L^2(0,T;U),\;\;z_0\in H,
	\end{equation}
	where $\Phi_t\in\mathcal{L}(L^2(0,T;U),H_{-1})$ is defined by
	\begin{equation}\label{inp-map}
		\Phi_tu:=\int_0^t \mathbb{T}_{-1}(t-s)Bu(s)ds.
	\end{equation}
	Notice that in the above formula, $\mathbb{T}_{-1}$ acts on $H_{-1}$ and the integration is carried out in $H_{-1}$. This motivates the following definition.
	\begin{definition}\label{defi:admcontrop}
		The operator $B\in\mathcal{L}(U,H_{-1})$ is called an \emph{admissible control operator}
		for $A$ if \text{Ran}$(\Phi_\tau)\subset H$ for some $\tau>0$ (hence, for all $\tau>0$).
	\end{definition}
	It is worth noting that if $B$ is an admissible control operator for $A$, then the closed graph theorem guarantees that $\Phi_t\in\mathcal{L}(L^2(0,T;U),H)$ for all $t\ge 0$. As a result, for any $u\in L^2(0,T;U)$ and $z_0\in H$, the solutions $z(\cdot)$ of the initial value problem  \eqref{eq:conpro1} stay in $H$ and form a continuous $H$-valued function of $t$. The operators $\Phi_t$ are commonly referred to as \emph{input maps} associated with the pair $(A,B)$, see \cite[Def.~2.1]{Weiss89}.

	\begin{definition}\label{defi:0controllable}
		Let $B\in\mathcal{L}(U,H_{-1})$ be an admissible control operator for $A$. Given $T>0$, the system \eqref{eq:conpro1} is said to be 
   \begin{enumerate}
        \item[$\bullet$] \emph{null-controllable in time $T$}, if for any $z_0\in H$ there exists a control $u\in L^2(0,T;U)$ such that  the solution of \eqref{eq:conpro1} satisfies $z(T)=0$.
        \item[$\bullet$] \emph{exactly controllable in time $T$}, if for any $z_0,\,z_1\in H$ there exists a control $u\in L^2(0,T;U)$ such that  the solution of \eqref{eq:conpro1} satisfies $z(T)=z_1$. 
   \end{enumerate}
	\end{definition}
It is easy to see that null controllability in time $T$ is equivalent to the following property: 
\begin{equation}\label{null-con}
	\mathrm{Ran}\, \mathbb{T}(T)\subset \mathrm{Ran}\, \Phi_T.
\end{equation}\\

 In the following we will deal with an abstract boundary control problem of the form
	\begin{align}\label{eq:boundconpro}
		\begin{cases}
			\dot{z}(t)=A_{\max}z(t), & t\in(0,T], \\
			Gz(t)=u(t), &t\in(0,T],\\
			z(0)=z_0\in H,
		\end{cases}
	\end{align}
	where $u\in L^2(0,T;U)$ is a control function. 
 \begin{assumptions}\label{ass:AmaxG}  
 Assume that $A_{\max}:\mcD(A_{\max})\subset H\to H $ is a closed linear operator, where the embedding $\mcD(A_{\max})\to H$ is a continuous linear map with dense range, 
and $G:\mcD(A_{\max})\to U$ is a linear operator such that the following assumptions are satisfied.
	\begin{enumerate}
		\item[\textbf{(A1)}] The operator $G:\mcD(A_{\max})\to U$ is surjective.
		\item[\textbf{(A2)}] The operator 
        \begin{equation}\label{eq:Agenerator}
            \begin{split}
                A&:=(A_{\max})|_{\mcD(A)},\\
                \mcD(A)&:=\left\{z\in\mcD(A_{\max})\colon Gz=0_H\right\}=\ker G
            \end{split}
        \end{equation} 
        generates a $C_0$-semigroup $\mathbb{T}:=(\mathbb{T}(t))_{t\geq0}$ on $H$.
	\end{enumerate} 
\end{assumptions}
 
 It is shown by Greiner \cite[Lemmas 1.2, 1.3]{Gr87} that under the above assumptions, the domain $\mcD(A_{\max})$ can be viewed as the direct sum of $ \mcD(A)$ and $\ker(\lambda-A_{\max})$ for any $\lambda\in\rho(A)$. Moreover, the operator $G|_{\ker(\lambda-A_{\max})}$ is invertible and the inverse
	\begin{equation}    
		\mathbb{D}_\lambda:=\left(G|_{\ker(\lambda-A_{\max})}\right)^{-1}: U\to \ker(\lambda-A_{\max})\subset H, \qquad \lambda\in\rho(A),
	\end{equation}
	is bounded. The operator $\mathbb{D}_\lambda$ is called the \emph{Dirichlet operator} associated with $A_{\max}$ and $G$.

    For a fixed $\lambda\in\rho(A)$ we consider the following \emph{control operator}
	\begin{equation}\label{eq:controlopB}
	    B:=(\lambda-A_{-1})\mathbb{D}_\lambda\in\mathcal{L}(U,H_{-1}).
	\end{equation}
  Using the resolvent identity, it is straightforward that the definition of $B$ is independent of the choice of $\lambda$.\\

  	As for any $u\in U$ and  $\lambda\in\rho(A)$ we have $\mathbb{D}_\lambda u\in \ker (\lambda-A_{\max})$ so $\lambda\mathbb{D_\lambda}u=A_{\max}\mathbb{D}_\lambda u$. Hence,
	\begin{align*}
		\left(A_{\max}-A_{-1}\right)\mathbb{D}_\lambda u=(\lambda-A_{-1})\mathbb{D}_\lambda u=Bu.
	\end{align*}
	Since $\mathbb{D}_\lambda$ is the inverse of $G|_{\ker(\lambda-A_{\max})}$ and $\mcD(A_{\max})=\mcD(A)\oplus \ker(\lambda-A_{\max})$, it follows that
	\begin{equation*}
		A_{\max}=(A_{-1}+BG)|_{\mcD(A_{\max})}.
	\end{equation*}
    Thus, by $Gz(t)=u(t)$, the system \eqref{eq:boundconpro} is equivalent to the problem \eqref{eq:conpro1}. We then have the following proposition; see, for instance, \cite[Rem.~10.1.4.]{TuWe09}.

\begin{proposition}\label{prop:boundcontequiv}
   If Assumptions \ref{ass:AmaxG} are satisfied, then the boundary control system \eqref{eq:boundconpro} can be reformulated as a control problem of the form \eqref{eq:conpro1}   with $A$ defined in \eqref{eq:Agenerator} and $B$ given by \eqref{eq:controlopB}.
\end{proposition}

	\section{Strong Feller property  of the transition semigroup}\label{sec:strongFeller}

We first aim to investigate the strong Feller property of the mild solution of the following formal general boundary noise problem on the separable Hilbert space $H$ with boundary space $U$:
\begin{equation}\label{eq:opAKmax1gen}
		\begin{cases}
			\dot{w}(t)=A_{\max}w(t), &t\in(0,T]\\
			Gw(t)=\dot\beta(t), &t\in(0,T],\\
			w(0)=z_0,
		\end{cases}
\end{equation}
where the process 
	\[(\beta(t))_{t\in [0,T]},\]
 is an $U$-valued Brownian motion with covariance operator $Q$.\\
 \begin{assumptions}\label{ass:noisecontrol}We assume the followings.\vspace{-0.5em}
 \begin{enumerate}
         \item For the operators $A_{\max}$ and $G$ and for the spaces $H$ and $U$ Assumptions \ref{ass:AmaxG} are satisfied, and $\dim U=n<\infty$, hence $Q$ is an $n\times n$ matrix.
         \item The generator $(A,\mcD(A))$ in \eqref{eq:Agenerator}  is self-adjoint, dissipative and has compact resolvent. In particular, $A$ generates a contractive analytic semigroup $\left(\mathbb{T}(t)\right)_{t\geq 0}$ on $H$.
         \item Analogously to \eqref{eq:controlopB}, define
         \begin{equation}\label{eq:Bdefgen}
		          B\coloneqq (1-A_{-1})\Dir_1\in\mathcal{L}\left(U,H_{-1}\right).
	   \end{equation}
         Using the analiticity of the semigroup, we obtain that the operator $\mathbb{T}_{-1}(t)B$ maps $U$ into $H$ for any $t>0$. Assume that
                \begin{equation}\label{eq:TBQHS}
                    \int_0^T\left\|\mathbb{T}_{-1}(t)BQ^{\frac{1}{2}}\right\|^2_{\mathrm{HS}(U,H)}dt<+\infty,
                \end{equation}
            where $\mathrm{HS}$ denotes the space of Hilbert--Schmidt operators between the appropriate spaces.         
     \end{enumerate}
 \end{assumptions}

Assumptions \ref{ass:noisecontrol} imply that $A$ has only point spectrum with eigenvalues $(-\lambda_k)_{k\in\nat}$ that form a sequence of nonpositive real numbers satisfying 
	\begin{align*}
		0\leq\lambda_0\leq\lambda_1\leq\cdots\leq\lambda_k\leq\lambda_{k+1}\leq\cdots
	\end{align*}
	and 
	\begin{equation}\label{eq:lambdak}
		\lambda_k\to+\infty,\quad k\to\infty.
	\end{equation}
	We may then choose a complete orthonormal set $(f_k)_{k\in\nat}\subset \mcD(A)$ of eigenfunctions such that
	\begin{equation}\label{eq:fk}
		A f_k=-\lambda_k f_k,\quad k\in\nat.
	\end{equation}
 Within this context, the scale of Hilbert spaces 
 \begin{align*}
     D(A)\subset H\subset H_{-1},
 \end{align*}
introduced in Section \ref{sec:control}, can be completed to a scale $(H_\alpha)_{\alpha\in\mathbb{R}}$ as follows. 

\begin{definition}\label{defi:Halpha}
    For every $\alpha\geq0$ and for any $\la>0$ we set \[H_\alpha:=\mcD((\lambda-A)^\alpha),\] equipped with its natural norm $    \|\cdot\|_\alpha$; that is, for $u\in H$ we have $u\in H_\alpha$ if and only if
\begin{equation}\label{eq:Halphanorm}
    \|u\|^2_\alpha:=\sum_{k\in\mathbb{N}}(\lambda-\lambda_k)^{2\alpha}|\langle u,f_k\rangle|^2<\infty.
\end{equation}
For $\alpha>0$, the space $H_{-\alpha}$ is defined as the dual of $H_\alpha$ with respect to the pivot space $H$. Equivalently, $H_{-\alpha}$ is the completion of $H$ for the norm
\begin{align*}
    \|u\|_{-\alpha}^2=\sum_{k\in\mathbb{N}}(\lambda-\lambda_k)^{-2\alpha}|\langle u,f_k\rangle|^2.
\end{align*}
It is well-known that the above norms are equivalent for all $\la>0$.
\end{definition} 
 
 \begin{remark}\label{rem:BadmissibleA}
        If Assumptions \ref{ass:noisecontrol} are satisfied, it follows from  \cite[Thm 2.2]{DPZ93} and \cite[Thm.~3.1]{FHR25} that the input map $\Phi_T$, defined in \eqref{inp-map}, is a Hilbert-Schmidt operator from $L^2(0,T;U)$ into $H$. Since every Hilbert-Schmidt operator is bounded, it follows that $B$ is an admissible control operator for $A$, see also \cite[Prop. 3.4]{FHR23} .
 \end{remark}
 
By Assumptions \ref{ass:noisecontrol}, using the terminology of \cite{DPZ93}, system \eqref{eq:opAKmax1gen} can be written in an It\^o form as
	\begin{equation}\label{eq:stocauchygen}
	\begin{cases}
		dw(t)=Aw(t)dt+Bd\beta(t), &t\in(0,T],\\
		w(0)=z_0,
	\end{cases}
	\end{equation}
and the weak solution $w(t)=w(t,z_0)\in H$, $t\in[0,T]$ of \eqref{eq:stocauchygen} is given by
\begin{equation}
w(t)=\mathbb{T}(t)z_0+\int_0^t\mathbb{T}_{-1}(t-s)Bd\beta(s).
\end{equation}

We mention that, for all $t\in(0,T]$, $w(t)$ is a Gaussian random variable with mean vector $\mathbb{T}(t)z_0$ and covariance operator $Q_t\in\mathcal{L}(H)$ given by
	\begin{align*}
		Q_tx=\int_0^t\mathbb{T}_{-1}(s)BQB^*\mathbb{T}^*(s)xds, \qquad \forall x\in \mcD(A^*)=\mcD(A).
	\end{align*}

Let $\mathcal{B}_b({H})$ (resp. $C_b({H})$) be the space of Borel measurable (resp. continuous) bounded functions from ${H}$ to $\mathbb{R}$. We shall denote by $\mathcal P_t$, $t\in[0,T]$, the \emph{transition semigroup} corresponding to problem \eqref{eq:stocauchygen}:
\begin{equation}\label{eq:transsgr}
\mathcal P_t\phi(z_0):=\mathbb{E}\left(\phi(w(t,z_0))\right), \qquad t\in[0,T],\;\; \phi\in \mathcal{B}_b({H}). 
\end{equation} 
\begin{definition}    
The transition semigroup $\mathcal P_t$ is said to be a \textit{strong Feller} semigroup at time $T>0$ if
\begin{equation}\label{eq:strongFellerprop}
\text{ for any }\phi\in \mathcal{B}_b({H}),\quad \mathcal P_T\phi\in C_b({H})\text{ holds}.    
\end{equation}
\end{definition}
It is well-known that $\mathcal P_t$ is strong Feller at time $T>0$ if and only if
\begin{align*}
	\text{Ran}\; \mathbb{T}(T)\subset \text{Ran}\;Q_T^{\frac{1}{2}},
\end{align*}
see for example \cite[Thm.~7.2.1]{DPZbook96}.\\

Let  $\Phi_t$, $t\in[0,T]$, be the input maps associated with $(A,B)$, see \eqref{inp-map}. 
It follows from \cite[Cor.~B3]{DPZbook} that
\begin{align*}
	\text{Ran}\; \Phi_TQ^{\frac{1}{2}}=\text{Ran}\; Q_T^{\frac{1}{2}}.
\end{align*}
\begin{remark}\label{rem:nullcontrFeller}
Taking in consideration \eqref{null-con}, we deduce that $\mathcal P_t$ is strong Feller at time $T>0$ if and only if the problem 
	\begin{equation}\label{eq:control}
	\left\{\begin{aligned}
		\dot{z}(t)&=Az(t)+BQ^{\frac{1}{2}}u(t),\; t\in(0,T],\\
		z(0)&=z_0,
	\end{aligned}\right.
\end{equation}
	is null controllable in time $T>0$ (see Definition \ref{defi:0controllable}).
\end{remark}

 In the following, we will investigate the null controllability of the above problem to obtain a necessary and sufficient condition for the strong Feller property of the transition semigroup.\\

By Assumptions \ref{ass:noisecontrol}, for any $z_0\in H$ and any $u\in L^2(0,T;U)$ the problem \eqref{eq:control} has a unique mild solution $z\in C([0,T],H)$ given by
		\begin{equation}\label{varformula}
			z(t)=\mathbb{T}(t)z_0+\int_{0}^t\mathbb{T}_{-1}(t-s)BQ^{\frac{1}{2}}u(s)ds=\mathbb{T}(t)z_0+\Phi_tQ^{\frac{1}{2}}u.
		\end{equation}

 \begin{proposition}\label{prop:0contrformula}
     Let Assumptions \ref{ass:noisecontrol} be satisfied. Let $(f_k)_{k\in\nat}\subset \mcD(A)$ be a set of linearly independent eigenfunctions of $A$ corresponding to $(\lambda_k)_{k\in\nat}$ forming a total subset in $H$. Then the problem \eqref{eq:control} is null controllable in time $T > 0$ if and only if for any $z_0\in H$ there is a control $u\in L^2\left(0,T;U\right)$ satisfying
	\begin{equation}\label{eq:0contrprf1}
		\e^{-\lambda_k T}\langle z_0,f_k\rangle_{H}=-\sum_{j=1}^n\left(Q^{\frac{1}{2}}B^*f_k\right)_{j}\int_0^T\e^{-\lambda_k (T-s)}u_j(s)\ds \text{ for all } k\in\nat.
	\end{equation}
\end{proposition}
\begin{proof}	
	By \eqref{eq:solcontrol} we know that for any $z_0\in H$ and  $u\in L^2(0,T;U)$, the mild solution $z(t)$ defined in \eqref{varformula} exists and is unique. We take any eigenfunction $f_k$ from the total subset satisfying \eqref{eq:fk}. Then
	\begin{align}
		\langle z(T),f_k\rangle_{H}&=\langle \mathbb{T}(T)z_0,f_k\rangle_{H} + \langle \Phi_TQ^{\frac{1}{2}} u,f_k\rangle_{H}\\
		&=\e^{-\lambda_k T}\langle z_0,f_k\rangle_{H} + \langle Q^{\frac{1}{2}}u, \Phi_T^* f_k \rangle_{L^2\left(0,T,U\right)}\\
		&=\e^{-\lambda_k T}\langle z_0,f_k\rangle_{H}+\int_0^T\e^{-\lambda_k (T-s)}\langle B^* f_k,Q^{\frac{1}{2}}u(s)\rangle_{U}\ds \\
		&=\e^{-\lambda_k T}\langle z_0,f_k\rangle_{H}+\int_0^T\e^{-\lambda_k (T-s)}\langle Q^{\frac{1}{2}}B^* f_k,u(s)\rangle_{U}\ds\\
		&=\e^{-\lambda_k T}\langle z_0,f_k\rangle_{H}+\sum_{j=1}^n\int_0^T\e^{-\lambda_k (T-s)}\left(Q^{\frac{1}{2}}B^*f_k\right)_{j} u_{j}(s)ds\\
		&=\e^{-\lambda_k T}\langle z_0,f_k\rangle_{H}+\sum_{j=1}^n\left(Q^{\frac{1}{2}}B^*f_k\right)_{j}\int_0^T\e^{-\lambda_k (T-s)}u_{j}(s)ds,
	\end{align}
where we have used that for every $T>0$ the adjoint $\Phi_T^*$ of the operator $\Phi_T$ is in $\mathcal{L}\left(H, L^2(0,T;U)\right)$ and is given by
 \begin{align}\label{eq:adjoint}
	\left(\Phi_T^*x\right)(t)=B^*\mathbb{T}^*(T-t)x, \qquad t\in[0,T],
\end{align}
for any $x\in \mcD(A^*)$, see \cite[Prop.~4.4.1]{TuWe09}.
 Since the closure of the linear hull of the functions $(f_k)$ is $H$, $z(T)=0$ is equivalent to the fact that $\langle z(T),f_k\rangle_{H}=0$ for each $k$. This implies that the problem is null controllable if and only if for any $z_0\in H$ there is a control $u\in L^2\left(0,T;U\right)$ such that \eqref{eq:0contrprf1} holds.
\end{proof}

\section{Strong Feller property of the solution of the network problem with Kirchhoff noise}\label{sec:strongFellernetwork}

 In this section we aim to investigate the strong Feller property of the transition semigroup of the network problem on the metric graph $\mG$ with Kirchhoff noise in the vertices, see \eqref{eq:stochnet},
 
	\begin{equation}\label{eq:stochnetKnoise}
		\left\{\begin{aligned}
			\dot{z}_{\me}(t,x) & =  (c_{\me} z_{\me}')'(t,x)-p_{\me}(x) z_{\me}(t,x), &x\in (0,1),\; t\in(0,T],\; \me\in\mE,\;\; & (a)\\
			0 & = I_{\mv}Z(t,\mv),\;   &t\in(0,T],\; \mv\in\mV,\;\;& (b)\\
			\dot{\beta}_{\mv}(t)& = C(\mv)^{\top}Z'(t,\mv), & t\in(0,T],\; \mv\in\mV,\;\;& (c)\\
			z_{\me}(0,x) & =  z_{0,\me}(x), &x\in [0,\ell_{\me}],\; \me\in\mE,\;\;& (d)
		\end{aligned}
		\right.
	\end{equation}
	where the process 
	\[(\beta(t))_{t\in [0,T]}=\left(\left(\beta_{\mv}(t)\right)_{t\in [0,T]}\right)_{\mv\in\mV},\] 
	is an $\real^{n}$-valued Brownian motion (Wiener process) with $|\mV|=n$ and covariance matrix 
	\begin{equation}\label{eq:kovmtx}
	   Q\in \real^{n\times n},
	\end{equation}
    and
\[\begin{split}
 0&< c_0\leq c_{\me}\in \mathrm{Lip}[0,\ell_{\me}],\quad\text{for some constant }c_0,\\
 0&\leq p_{\me}\in L^{\infty}(0,\ell_{\me}),\quad \me\in\mE.
\end{split}\]

We start with the deterministic problem on the network:
	\begin{equation}\label{eq:netcp}
		\left\{\begin{aligned}
			\dot{z}_{\me}(t,x) & =  (c_{\me} z_{\me}')'(t,x)-p_{\me}(x) z_{\me}(t,x), &x\in (0,\ell_{\me}),\;t\in(0,T],\; \me\in\mE,\;\; & (a)\\
			0 & = I_{\mv}Z(t,\mv),\;   &t\in(0,T],\; \mv\in\mV,\;\;& (b)\\
			0 & =  C(\mv)^{\top}Z'(t,\mv), & t\in(0,T],\; \mv\in\mV,\;\;& (c)\\
			z_{\me}(0,x) & =  z_{0,\me}(x), &x\in [0,\ell_{\me}],\;\; \me\in\mE.\;\;& (d)
		\end{aligned}
		\right.
	\end{equation}
	where $0$ denotes the constant $0$ vector of dimension $d_{\mv}-1$ on the left-hand-side of $(\ref{eq:netcp}b)$. Recall that equations $(\ref{eq:netcp}b)$ assume continuity of the function $z(t,\cdot)$, cf.~\eqref{eq:contv}, and equations $(\ref{eq:netcp}c)$ are called \emph{Kirchhoff conditions} which become standard Neumann boundary conditions in vertices of degree $1$.
 
 We show that this system can be written in the form of \eqref{eq:boundconpro} with $u(t)=0$. First we consider the Hilbert space
	\begin{equation}\label{eq:E2}
		\mcH\coloneqq \prod_{\me\in\mE}L^2\left(0,\ell_{\me}\right)
	\end{equation}
	as the \emph{state space} of the edges, endowed with the natural inner product
	\[\langle f,g\rangle_{\mcH}\coloneqq \sum_{\me\in\mE} \int_0^{\ell_{\me}} f_{\me}(x)g_{\me}(x) dx,\qquad
	f=\left(f_{\me}\right)_{\me\in\mE},\;g=\left(g_{\me}\right)_{\me\in\mE}\in \mcH.\]
	On $\mcH$ we define the operator
	\begin{equation}\label{eq:opAmax}
		A_{\max}\coloneqq \diag\left(\frac{d}{dx}\left(c_{\me} \frac{d}{dx}\right)-p_{\me} \right)_{\me\in\mE}
	\end{equation}
	with maximal domain
	\begin{equation}\label{eq:domAmax}
		\mcD(A_{\max}) =\left\{z\in\mcH^2\colon I_{\mv}Z(\mv)=0,\; \mv\in\mV \right\},
	\end{equation}
	where
	\[\mcH^2\coloneqq \prod_{\me\in\mE} H^2(0,\ell_{\me})\]
	with $H^2(0,\ell_{\me})$ being the Sobolev space of twice weakly differentiable $L^2(0,\ell_{\me})$-functions having their first and second weak derivatives in $L^2(0,\ell_{\me})$.
	We also introduce the \emph{boundary space} 
	\begin{equation}\label{eq:Y}
		\mcY\coloneqq \ell^2(\real^{n})\cong \real^{n}.
	\end{equation}

	Notice that for fixed $z\in \mcD(A_{\max})$, the boundary (or vertex) conditions can be written as
	\begin{equation}\label{eq:contKir}
		0_{\real^{d_{\mv}-1}}=I_{\mv}Z(\mv),\; 0=C(\mv)^{\top}Z'(\mv),\quad \mv\in\mV,
	\end{equation}
	cf.~$(\ref{eq:netcp}b)$ and $(\ref{eq:netcp}c)$.
	
	\begin{remark}
		Define $A_{\mv}$ as the square matrix that arises from $I_{\mv}$ by inserting an additional row containing only $0$'s; that is,
		\begin{equation}\label{eq:Av}
			A_{\mv}=\begin{pmatrix}
				1 & -1 & & \\
				& \ddots & \ddots & \\
				& & 1 & -1 \\
				0 & \hdots & 0 & 0
			\end{pmatrix}\in\real^{d_{\mv}\times d_{\mv}}.
		\end{equation}
		Furthermore, let $B_{\mv}$ be the square matrix defined by
		\begin{equation}\label{eq:Bv}
			B_{\mv}=\begin{pmatrix}
				0 & \hdots &  0\\
				\vdots   &  & \vdots\\
				0 & \hdots &  0\\
				c_{\me_1}(\mv) & \hdots &  c_{\me_{d_{\mv}}}(\mv),
			\end{pmatrix}\in\real^{d_{\mv}\times d_{\mv}}.
		\end{equation} 
		It is straightforward that for a fixed $\mv\in\mV$, equations \eqref{eq:contKir} have the form
		\begin{equation}\label{eq:AvBv}
			A_{\mv}Z(\mv)+B_{\mv}Z'(\mv)=0_{\real^{d_{\mv}}},
		\end{equation}
		where the $d_{\mv}\times (2d_{\mv})$ matrix $\left(A_{\mv},B_{\mv}\right)$ has maximal rank. Thus, our vertex conditions have the form as in \cite[Sec.~1.4.1]{BeKu}.
	\end{remark}

	We now define the boundary operator $G\colon \mcD(A_{\max})\to \mcY$ by
	\begin{equation}\label{eq:opG}
		\begin{split}
			Gz\coloneqq \left(C(\mv)^{\top}Z'(\mv)\right)_{\mv\in\mV}.
		\end{split}
	\end{equation}

	With these notations, we can finally rewrite \eqref{eq:netcp} in form of 
 \begin{align}\label{eq:opAKmaxdeterm}
		\begin{cases}
			\dot{z}(t)=A_{\max}z(t), &t\in(0,T]\\
			Gz(t)=0, &t\in(0,T]\\
			z(0)=z_0,
		\end{cases}
	\end{align}
 cf.~\eqref{eq:boundconpro}. Define 
 \begin{equation}\label{eq:amain}
     A\coloneqq A_{\max}|_{\mcD(A)},\; \mcD(A)=\ker G
 \end{equation} 
 see Assumptions \ref{ass:AmaxG}. We claim now that we are in the situation of Section \ref{sec:control} and \ref{sec:strongFeller}.

\begin{proposition}\label{prop:assumptionfornet}
    For the spaces $H=\mcH$ and $U=\mcY$ defined in \eqref{eq:E2}, \eqref{eq:Y} and for the operators $A_{\max}$ defined in \eqref{eq:opAmax}, \eqref{eq:domAmax} and $G$ defined in \eqref{eq:opG} Assumptions \ref{ass:noisecontrol} are satisfied.
\end{proposition}
\begin{proof}
    The surjectivity of $G$ follows by \cite[Prop.~A.1]{KS23}. Assumptions \ref{ass:noisecontrol}(1) and (2) follow then by \cite[Prop.~2.3]{KS23} and \cite[Prop.~3.6]{KS23}. Assumption (3) is a consequence of \cite[Thm.~3.3]{KS23}, c.f. Proposition \ref{prop:solH14}.
\end{proof}

 The above proposition implies that the bounded operators
    \begin{equation}\label{eq:Dirop}
	    \mathbb{D}_\lambda:=\left(G|_{\ker(\lambda-A_{\max})}\right)^{-1}: \mathcal{Y}\to \ker(\lambda-A_{\max})\subset \mcH, \; \lambda\in(0,+\infty)
	\end{equation}
    and
    \begin{equation}\label{eq:Bdef}
		B\coloneqq (1-A_{-1})\Dir_1\in\mathcal{L}\left(\mathcal{Y},\mathcal{H}_{-1}\right),
	\end{equation}
 can be defined, and by Remark \ref{rem:BadmissibleA}, $B$ is an admissible control operator for $A$, where $A$ is the operator in \eqref{eq:amain}.\\

It is shown in \cite[(3.11)]{KS23} that for the space $\mcH_{\frac12}$, see Definition \ref{defi:Halpha}, we have
\begin{equation}\label{eq:H12def}
    \mcH_{\frac12}\cong\left\{z\in \mcH^1\colon I_{\mv}Z(\mv)=0,\; \mv\in\mV\right\},    
\end{equation}
where
\begin{equation}
\mcH^1\coloneqq \prod_{\me\in\mE}H^1(0,\ell_{\me}).
\end{equation}
Here $H^1(0,\ell_{\me})$ is the standard Sobolev space of absolutely continuous functions on the edge $\me$ having square integrable weak derivatives. That is, $\mcH_{\frac12}$ is isomorphic to the space of absolutely continuous functions on the edges satisfying the continuity condition in the vertices.

We now introduce the boundary operator $L\colon \mcD(L)\to \mcY$ defined by
		\begin{equation}\label{eq:L}
			\begin{split}
				\mcD(L) & = \mcH_{\frac12},\\
				L z & \coloneqq\left(z(\mv)\right)_{\mv\in\mV}\in \mathcal{Y},
			\end{split}
		\end{equation}
		where $z(\mv)$ denotes the common vertex value of the function $z\in \mcD(L)$ on the edges incident to $\mv$. That is, $L$ assigns to each function $z$ that is continuous on $\mG$ the vector of the vertex values of $z$. Observe that $\mcD(A_{\max})\subset \mcD(L)$ holds. In the following claim, the adjoint of $B$ which is a bounded operator from $\mcD(A^*)=\mcD(A)\cong \mcH_1$ to $\mcY$, is identified with its extension to $\mcH_{\frac12}$, where we use that $\mcH_1\subset\mcH_{\frac12}$ with an embedding $\mcH_1\to \mcH_{\frac12}$ being a continuous linear map with dense range.
		
	\begin{proposition}\label{prop:Bstar}
		For the adjoint of the operator $B$ defined in \eqref{eq:Bdef}
		\begin{equation}\label{Bstar}
		    B^*   =-L\text{ on }\mcH_{\frac12}.
		\end{equation}
	\end{proposition}
	\begin{proof}
		Let $\alpha\in \mathcal{Y}$ and $h\in \mcD(A^*)=\mcD(A).$ Using the fact that $A$ is self-adjoint and \cite[Prop.~2.3]{KS23}, we obtain
		\begin{align}
			&\langle B\alpha, h \rangle_{\mcH_{-1},\mcD(A)}=\langle (1-A_{-1})\Dir_1 \alpha, h \rangle_{\mcH_{-1},\mcD(A)}=\langle \Dir_1\alpha, (1-A^*)h \rangle_{\mcH}\notag\\
            &=\langle \Dir_1 \alpha, (1-A)h \rangle_{\mcH}\notag\\
			&=\langle \Dir_1 \alpha, h \rangle_{\mcH}+\langle \Dir_1 \alpha, (-A)h \rangle_{\mcH}\notag\\
			&=\sum_{\me\in\mE} \int_0^{\ell_{\me}} (\Dir_1 \alpha)_{\me}(x)\cdot h_{\me}(x)\dx+\sum_{\me\in\mE} \int_0^{\ell_{\me}} c_{\me}(x)\cdot (\Dir_1 \alpha)'_{\me}(x)\cdot h'_{\me}(x)\dx \label{eq:admobsbiz1}\\
			&+\sum_{\me\in\mE} \int_0^{\ell_{\me}} p_{\me}(x)\cdot (\Dir_1 \alpha)_{\me}(x)\cdot h_{\me}(x)\dx.\notag
		\end{align}
		
		We now integrate by parts in the second term of \eqref{eq:admobsbiz1} and obtain
		\begin{align*}
			\sum_{\me\in\mE} \int_0^{\ell_{\me}} c_{\me}(x)\cdot (\Dir_1 \alpha)'_{\me}(x)\cdot h'_{\me}(x)\dx&=\sum_{\me\in\mE} \left[ c_{\me}(x)\cdot (\Dir_1 \alpha)'_{\me}(x)\cdot h_{\me}(x)\right]_0^{\ell_{\me}}\\
			&-\sum_{\me\in\mE} \int_0^{\ell_{\me}} \left(c_{\me}(x)\cdot (\Dir_1 \alpha)'\right)'_{\me}(x)\cdot h_{\me}(x)\dx.
		\end{align*}
		Using $\Dir_1 \alpha\in\Ker(1-A_{\max})$, we have that
		\[(\Dir_1 \alpha)_{\me}-\left(c_{\me}\cdot (\Dir_1 \alpha)'\right)'_{\me}+p_{\me}\cdot (\Dir_1 \alpha)_{\me}=0\text{ for all }\me\in\mE,\]
		thus
		\begin{align*}
			\langle B\alpha, h \rangle_{\mcH}&=\sum_{\me\in\mE} \left[ c_{\me}(x)\cdot (\Dir_1 \alpha)'_{\me}(x)\cdot h_{\me}(x)\right]_0^{\ell_{\me}}.
		\end{align*}

		Since $h$ is continuous on $\mG$, we obtain
		\begin{align*}
			\sum_{\me\in\mE} \left[ c_{\me}(x)\cdot (\Dir_1 \alpha)'_{\me}(x)\cdot h_{\me}(x)\right]_0^{\ell_{\me}}&=-\sum_{\mv\in\mV}\sum_{\me\in\mE_{\mv}}c_{\me}(\mv)\cdot (\Dir_1 \alpha)'_{\me}(\mv)\cdot h_{\me}(\mv)\\
			&=-\sum_{\mv\in\mV}h(\mv)\sum_{\me\in\mE_{\mv}}c_{\me}(\mv)\cdot (\Dir_1 \alpha)'_{\me}(\mv)\\
			& =-\sum_{\mv\in\mV}h(\mv)\cdot\left(G\Dir_1 \alpha\right)_{\mv}\\
			&=-\sum_{\mv\in\mV}h(\mv)\cdot \left(\alpha\right)_{\mv}=\langle \alpha ,- L  h\rangle_{\mathcal{Y}},
		\end{align*}
		where we have used that, by definition, $G\Dir_1 \alpha= \alpha$. Combining all the above facts, we have that
		\begin{equation}
			\langle B\alpha, h \rangle_{\mcH_{-1},\mcD(A)}=\langle \alpha ,-L  h\rangle_{\mathcal{Y}}.
		\end{equation}
		Thus $B^*=- L$ holds on $\mcD(A)$. Using the density of $\mcD(A)\cong \mcH_1$ in $\mcH_{\frac12}$ we obtain
		\begin{equation}
			B^*   =-L\text{ on }\mcH_{\frac12}.
		\end{equation}
	\end{proof}

By the previous consideration for the deterministic version of the problem, system \eqref{eq:stochnetKnoise} can be reformulated in $\mathcal{H}$ as follows:
		\begin{equation}\label{eq:opAKmax1}
		\begin{cases}
			\dot{w}(t)=A_{\max}w(t), &t\in(0,T]\\
			Gw(t)=\dot\beta(t), &t\in(0,T],\\
			w(0)=z_0,&
		\end{cases}	
		\end{equation}
cf.~\eqref{eq:opAKmax1gen}. Hence, as seen in Section \ref{sec:strongFeller}, the system \eqref{eq:opAKmax1} can be written in an It\^o form as
	\begin{equation}\label{eq:stocauchy}
	\begin{cases}
		dw(t)=Aw(t)dt+Bd\beta(t), &t\in(0,T],\\
		w(0)=z_0,
	\end{cases}
	\end{equation}
    see \eqref{eq:stocauchygen}. On the other hand, it follows from \cite[Thm.~3.5]{KS23} that for any $z_0\in \mathcal{H}$, problem \eqref{eq:stocauchy} has a unique mild solution $w(t)=w(t,z_0)\in \mathcal{H}$, $t\in[0,T]$, given by
\begin{equation}\label{eq:ZKt}
    w(t)=\mathbb{T}(t)z_0+\int_0^t\mathbb{T}_{-1}(t-s)Bd\beta(s),
\end{equation}
where $\left(\mathbb{T}(t)\right)_{t\geq 0}$ is the contractive analytic semigroup on $\mcH$ generated by $A$, see Proposition \ref{prop:assumptionfornet}.
We define the transition semigroup of the solutions \eqref{eq:ZKt} as in \eqref{eq:transsgr}, that is, 
\begin{equation}\label{eq:transsgrnet}
\mathcal P_t\phi(z_0):=\mathbb{E}\left(\phi(w(t,z_0))\right), \qquad t\in[0,T],\;\; \phi\in \mathcal{B}_b(\mathcal{H}). 
\end{equation} 

In the same way as in Section \ref{sec:strongFeller} in Remark \ref{rem:nullcontrFeller}, one observes that $(\mathcal P_t)_{t\in[0,T]}$ is strong Feller at time $T>0$ if and only if the problem
	\begin{equation}\label{eq:netcontrol}
	\left\{\begin{aligned}
		\dot{z}(t)&=Az(t)+BQ^{\frac{1}{2}}u(t),\; t\in(0,T],\\
		z(0)&=z_0,
	\end{aligned}\right.
\end{equation}
	is null controllable at time $T>0$, see \eqref{eq:control}, where the operators are defined in \eqref{eq:amain}, \eqref{eq:Bdef} and \eqref{eq:kovmtx}.\\

From now on we assume that the diffusion coefficients in \eqref{eq:stochnetKnoise} are $c_{\me}=1$, $\me\in\mE$. In the following we give a sufficient condition for the null controllability of \eqref{eq:netcontrol}, or, equivalently, the strong Feller property of the transition semigroup defined in \eqref{eq:transsgrnet} to hold. By Proposition \ref{prop:assumptionfornet} there exists an increasing sequence of non-negative real numbers
\begin{equation}\label{eq:eigenvaluesofA}		  
0\leq \lambda_0\leq\lambda_1\leq\cdots\leq\lambda_k\leq\lambda_{k+1}\leq\cdots,\la_k\to +\infty,
\end{equation}
and a set $(f_k)_{k\in\nat}\subset \mcD(A)$ of eigenfunctions such that
	\begin{equation}\label{eq:fknet}
		Af_k=-\lambda_k f_k,\quad k\in\nat,
	\end{equation}
	where the functions $(f_k)_{k\in\nat}$ form a complete orthornomal system in $\mcH$. By \cite[(2.2) and Thm.~4.3]{BKS18} we have that there exist constants $l_1,\, l_2>0$ such that
\begin{equation}\label{eq:lambdaknet}
l_1\cdot k^2\leq -\lambda_k\leq l_2\cdot k^2,\quad k\in\nat.
\end{equation}
If the (algebraic) multiplicity of an eigenvalue is $d$, then for some $k$, $\lambda_k=\lambda_{k+1}=\cdots =\lambda_{k+d-1}$ holds (where $d=1$ can happen). Therefore, we also consider the strictly increasing subsequence of $(\la_k)$ consisting of distinct eigenvalues of $A$,  denoted by $(\mu_k)$; that is, the sequence
	\begin{equation}\label{eq:muks}
	0\leq \mu_{0}<\mu_{1}<\cdots <\mu_{k}<\cdots,\quad k\in\nat,
	\end{equation}
	where $-\mu_k$, $k\in\nat$ are the distinct eigenvalues of $A$. Since the multiplicity of eigenvalues are uniformly bounded (see, for example, \cite{BKS18}) it follows from \eqref{eq:lambdaknet} that for some constants $\tilde{l}_1,\,\tilde{l}_2>0$,
\begin{equation}\label{eq:mukasymp}
\tilde{l}_1\cdot k^2\leq -\mu_{k}\leq \tilde{l}_2\cdot {k}^2,\quad k\in\nat
\end{equation}
holds. 
 It is well-known that $\mu_0=0$ and the corresponding eigenspace is $1$ dimensional and is spanned by the constant $1$ function on $\mG$.	\\

We are going to prove  the strong Feller property of the transition semigroup $\mathcal{P}_t$ defined in \eqref{eq:transsgrnet}, that is, the null-controllability of problem \eqref{eq:netcontrol} in the case $\mG$ is a tree and the control acts in all boundary vertices, except for, possibly, one of them. 
\begin{definition}\label{defi:treeGamma}
   In a graph $\mG$ we call a vertex of degree $1$ a \emph{boundary vertex}, and denote the set of boundary vertices by $\Gamma.$ A \emph{path} in a graph is an alternating sequence of vertices and edges $(\mv_n, \me_n, \mv_{n-1}, \me_{n-1},\dots , \mv_1, \me_1, \mv_0)$ with no repeated edges or vertices.  $\mG$ is a \emph{tree} if it is connected and contains no cycle, or, equivalently, if for any two different vertices $\mv$ and  $\mv'$, there is a unique path $P(\mv,\mv')$ in $\mG$ with starting point $\mv$ and finishing point $\mv'$.
\end{definition}

For the proof we will need the following notions which we cite from \cite[Def.~I.1.13, I.1.16.]{AvdIvbook}.
\begin{definition}
    A system $E=(\xi_k)_{k\in\nat}$ is \emph{minimal}, if for any $j$, element $\xi_j$ does not belong to the closure of the linear hull of all the remaining elements.\newline  
    A system $E=(\xi_k)_{k\in\nat}$ is said to be an $\mathcal{L}$-basis, if $E$ is an image of an isomorphic mapping of some orthonormal family.
\end{definition}

   It is straightforward that by Proposition \ref{prop:boundcontequiv}, system \eqref{eq:netcontrol} can be written in the following equivalent form:
\begin{equation}\label{eq:heat}
 		\left\{\begin{aligned}
 			\dot{z}_{\me}(t,x) & = z_{\me}''(t,x)-p_{\me}(x) z_{\me}(t,x), &x\in (0,\ell_e),\; t\in(0,T],\; \me\in\mE,\;\; \\
 			0 & = I_{\mv}Z(t,\mv),\;   &t\in(0,T],\; \mv\in\mV,\;\;\\
 			Q^{\frac{1}{2}}u_\mv(t)& = \sum_{e\in E_\mv}z'_e(t,\mv), & t\in(0,T],\; \mv\in\mV,\;\;\\
 			z_{\me}(0,x) & =  z_{0,\me}(x),  &x\in [0,\ell_{\me}],\; \me\in\mE.\;\;
 		\end{aligned}
 		\right.
 	\end{equation}

\begin{theorem}\label{thm:main}
    Let $\mG$ be a tree, $c_{\me}=1$. If the covariance matrix $Q=\diag(q_{\mv})_{\mv\in\mV}$ in  \eqref{eq:kovmtx} is diagonal, and $q_{\mv}\neq 0$ for all boundary vertices except for, possibly, one of them, where it may be zero, then  the transition semigroup $\mathcal{P}_t$ defined in \eqref{eq:transsgrnet} is strong Feller at any time $T>0$.
\end{theorem}
\begin{proof}
We will deduce the result from the exact controllability of the a boundary control problem for a wave equation on a network:
  \begin{equation}\label{eq:wave}
 		\left\{\begin{aligned}
 			\ddot{z}_{\me}(t,x) & =z_{\me}''(t,x)-(p_{\me}(x)+1) \cdot z_{\me}(t,x), &  x\in (0,\ell_e),\; t\in(0,T],\; \me\in\mE\;\; \\
 			0 & = I_{\mv}Z(t,\mv),\;   &t\in(0,T],\; \mv\in\mV,\;\;\\
 			Q^{\frac{1}{2}}u_\mv(t)& = \sum_{e\in E_\mv}z'_e(t,\mv), & t\in(0,T],\; \mv\in\mV,\;\;\\
 			z_{\me}(0,x) & =  \dot z_{\me}(0,x)=0 &x\in [0,\ell_{\me}],\; \me\in\mE.\;\;
 		\end{aligned}
 		\right.
 	\end{equation}
 Defining
    \[\Gamma_1:=\Gamma\setminus\{\gamma_1\}\]
for $\gamma_1\in\Gamma$ arbitrary and $J^*=\emptyset$, we obtain by Proposition \ref{prop:STtree} that $\{\Gamma_1,J^*\}$ is an ST active set of $\mG$ in the sense of Definition \ref{defi:STactive}, that is of \cite[Def.~3.2]{AvdZhao22}. Hence, we can apply \cite[Thm.~3.14]{AvdZhao22} to obtain that the system \eqref{eq:wave} is exactly controllable on $\mcH_{\frac12}$. Furthermore, for the corresponding operator $A- Id$ there is a strictly positive sequence 
\[0<\tilde{\la}_0\leq \tilde{\lambda}_1\leq \cdots\leq \tilde{\lambda}_k\leq\tilde{\la}_{k+1}\leq\cdots ,\tilde{\la}_k\to +\infty,\]
with
\[\tilde{\la}_k=\la_k+ 1,\quad k\in\nat \]
such that 
\[(A-  Id)f_k=-\tilde{\la}_k f_k,\quad k\in\nat,\]
where $(-\la_k)$ are the eigenvalues of $A$, and the eigenfunctions $f_{k},$ $k\in\nat$ form a complete orthonormal system in $\mcH$, see \eqref{eq:eigenvaluesofA} and \eqref{eq:fknet}. 

We will use the method introduced in \cite{AvdMikh08} to show that \eqref{eq:netcontrol}, that is \eqref{eq:heat} with first equations \begin{equation}\label{eq:heatplusla1}
    \dot{z}_{\me}(t,x) =z_{\me}''(t,x)-(p_{\me}(x)+1) \cdot z_{\me}(t,x),\; x\in (0,\ell_e),\; t\in(0,T],\; \me\in\mE
    \end{equation} 
    is null-controllable at any time $\tau>0$, for any $z_0\in\mcH_{-\frac{1}{2}}$. For this purpose we introduce the spectral data
\begin{equation}
    \alpha_k\coloneqq \left(\frac{f_k(\mv)}{\sqrt{\tilde{\la}_k}}\right)_{\mv\in\Gamma_1}\in\real^m,\quad k\in\nat,
\end{equation}
where $m=|\Gamma_1|$ is the number of the vertices in $\Gamma_1$.
By \cite[Thm.~III.3.10]{AvdIvbook} we know that the exactly controllability of  system \eqref{eq:wave} at time $T$ is equivalent to the fact that
\begin{equation}
    E_{\pm k}(t)\coloneqq \alpha_k\cdot\e^{\pm i\sqrt{\tilde{\la}_k}t},\quad  k\in\nat
\end{equation}
forms an $\mathcal{L}$-basis in $L^2\left(0,T;\comp^m\right)$. From \cite[Thm.~II.5.20]{AvdIvbook} we obtain that the system
\begin{equation}
    Q_k(t)\coloneqq \alpha_k\cdot \e^{-\tilde{\la}_kt},\quad  k\in\nat
\end{equation}
is minimal in $L^2\left(0,\tau;\real^m\right)$ for any $\tau>0$. Furthermore, for the corresponding bi-orthogonal family
\begin{equation}
    Q_k'\in L^2(0,\tau;\real^m),\quad  k\in\nat
\end{equation}
with
\begin{equation}\label{eq:Qkbiorthogonal}
    \langle Q_{\ell},Q'_k\rangle_{L^2(0,\tau;\real^m)}\coloneqq 
\sum_{\mv\in\Gamma_1}\int_0^{\tau} (Q_k(s))_{\mv}\cdot Q'_k(s)_{\mv} \ds =\delta_{k,\ell},\quad k,\ell\in\nat,
\end{equation}
there exist positive constants $C(\tau)$ and $\beta$ such that
\begin{equation}
  \left\|Q'_k\right\|_{L^2(0,\tau;\real^m)}\leq C(\tau)\cdot \left\|E'_k\right\|_{L^2(0,\tau;\comp^m)}\e^{\beta \sqrt{\tilde{\la}_k}},\quad k\in\nat
\end{equation}
for any $\tau>0$. Since $\{E_{\pm k}\}_{k\in\nat}$ is an $\mathcal{L}$-basis in $L^2\left(0,T;\comp^m\right)$, there exists $c>0$ such that
\[\left\|E'_k\right\|_{L^2(0,T;\comp^m)}\leq c,\quad k\in\nat\]
holds, see \cite[page 26]{AvdIvbook}. Together with \eqref{eq:lambdaknet} and the fact that $z_0\in \mcH_{-\frac{1}{2}}$ we obtain that for the control defined as
\begin{equation}\label{eq:defiutilde}
    \tilde{u}(s)=\sum_{k=0}^{\infty}\frac{\langle z_0,f_k\rangle_{\mcH}}{\sqrt{\tilde{\la}_k}}\e^{-\tilde{\la}_k T}Q_k'(T-s) \in L^2\left(0,T;\real^m\right)
\end{equation}
is satisfied. Furthermore, from \eqref{eq:Qkbiorthogonal} follows that 
\begin{equation}\label{eq:momproblQk}
    \frac{\langle z_0,f_k\rangle_{\mcH}}{\sqrt{\tilde{\la}_k}}\cdot \e^{-\tilde{\la}_k T}=\langle Q_k,\tilde{u}^{T}\rangle_{L^2(0,T;\real^m)}, \quad  k\in\nat
\end{equation}
with $\tilde{u}^{T}(s)=\tilde{u}(T-s)$ holds.

By Propositions \ref{prop:0contrformula} and \ref{prop:Bstar} we know that system \eqref{eq:heat} with first equations \eqref{eq:heatplusla1} is null-controllable  at time $T>0$, for $z_0\in\mcH_{-\frac{1}{2}}$  if and only if  there is a control $\tilde{y}\in L^2(0,T;\real^n)$ satisfying the moment problem
\begin{equation}\label{eq:momproblnullcontr}
		\frac{\langle z_0,f_k\rangle_{\mcH}}{\sqrt{\tilde{\la}_k}}\cdot \e^{-\tilde{\la}_k T}=-\sum_{\mv\in\mV}\frac{\sqrt{q_{\mv}}\cdot f_k(\mv)}{\sqrt{\tilde{\la}_k}}\int_0^{T}\e^{-\tilde{\la}_k (T-s)}\tilde{y}_{\mv}(s)\ds, \quad  k\in\nat.
	\end{equation}
Define
\begin{equation}
    \tilde{y}_{\mv}(s)=
    \begin{cases}
        \frac{1}{\sqrt{q_{\mv}}}\tilde{u}_{\mv}(s), & \mv\in\Gamma_1,\\
        0, & \mv\notin\Gamma_1
    \end{cases}
\end{equation}
for $s\in[0,T]$. Then by \eqref{eq:defiutilde}, $\tilde{y}\in L^2(0,T;\real^n)$ holds. Taking inner product, from \eqref{eq:momproblQk} and the assumption on $Q$ follows that \eqref{eq:momproblnullcontr} is satisfied by $\tilde{y}$.\\

Turning back to the original problem, we have to show the null-controllability of problem \eqref{eq:netcontrol}, that is of \eqref{eq:heat}. For a given $z_0\in\mcH_{-\frac{1}{2}}$, define
\begin{equation}
    u(s)=\sum_{k=0}^{\infty}\frac{\langle z_0,f_k\rangle_{\mcH}}{\sqrt{\tilde{\la}_k}}\cdot \e^{-({\la}_k+1)T+ s} Q_k'(s) \in L^2\left(0,T;\real^m\right),
\end{equation}
where $\left(Q'_k(t)\right)$ is the biorthogonal family introduced above, and let
\begin{equation}\label{eq:controlheat}
    y_{\mv}(s)=
    \begin{cases}
        \frac{1}{\sqrt{q_{\mv}}} u_{\mv}(s), & \mv\in\Gamma_1,\\
        0, & \mv\notin\Gamma_1
    \end{cases}
\end{equation}
for $s\in[0,T]$.
Let $\ell\in\nat$ arbitrary. Then  we have
\begin{equation}
\begin{split}
    & - \sum_{\mv\in\mV}\sqrt{q_{\mv}}\cdot f_{\ell}(\mv)\int_0^{T}\e^{-\la_{\ell} (T-s)}y_{\mv}(s)\ds=\sum_{\mv\in\mV} \sqrt{q_{\mv}}\cdot f_{\ell}(\mv)\int_0^{T}\e^{-\la_{\ell} s}y_{\mv}(T-s)\ds\\ 
    &=\sum_{\mv\in\Gamma_1}f_{\ell}(\mv)\int_0^{T}\e^{-\la_{\ell} s }\sum_{k=0}^{\infty}\frac{\langle z_0,f_k\rangle_{\mcH}}{\sqrt{\tilde{\la}_k}}\cdot \e^{{-\la}_kT- s} (Q_k'(s))_{\mv}\ds\\
    &=\sum_{k=0}^{\infty}\frac{\langle z_0,f_k\rangle_{\mcH}}{\sqrt{\tilde{\la}_k}}\cdot \e^{{-\la}_kT}\sum_{\mv\in\Gamma_1}f_{\ell}(\mv)\int_0^{T}\e^{-\tilde{\la}_{\ell} s } (Q_k'(s))_{\mv}\ds\\
    &=\sum_{k=0}^{\infty}\frac{\langle z_0,f_k\rangle_{\mcH}}{\sqrt{\tilde{\la}_k}}\cdot \e^{{-\la}_kT}\sqrt{\tilde{\la}_{\ell}}\cdot \langle Q_{\ell},Q'_k\rangle_{L^2(0,T;\real^m)}\\
    &=\sum_{k=0}^{\infty}\langle z_0,f_k\rangle_{\mcH}\cdot \e^{{-\la}_kT}\cdot \delta_{k,\ell}=\langle z_0,f_{\ell}\rangle_{\mcH}\cdot \e^{-\la_{\ell} T}.
\end{split}
	\end{equation}
Thus, we obtained that  for any $z_0\in\mcH_{-\frac{1}{2}}$ and $T>0$, with control defined in \eqref{eq:controlheat},
\begin{equation}
    \langle z_0,f_{\ell}\rangle_{\mcH}\cdot \e^{-\la_{\ell} T}=-\sum_{\mv\in\mV}\sqrt{q_{\mv}}\cdot f_{\ell}(\mv)\int_0^{T}\e^{-\la_{\ell} (T-s)}y_{\mv}(s)\ds,\quad \ell\in\nat
\end{equation}
is satisfied. Hence, by Proposition \ref{prop:0contrformula}, system
\eqref{eq:heat} is null-controllable at any time $T>0$ for any $z_0\in\mcH_{-\frac{1}{2}}$ hence for   for any $z_0\in\mcH$
and equivalently, the transition semigroup $\mathcal{P}_t$ is strong Feller at any time $T>0$.
\end{proof}

Example \ref{ex:eqst} below shows that, in general, the above theorem cannot be strengthened.

\begin{definition}
    We say that a tree $\mG$ is a \emph{star graph} if it has one central vertex called $\mv_c$, boundary vertices $\mv_1,\dots ,\mv_N$, and the edges of $\mG$ connect the central vertex to the boundary vertices, see Figure \ref{fig:stargraph}.

\begin{figure}[!ht]
\centering
\begin{tikzpicture}
    \node[circle,fill=black] at (360:0mm) (center) {};
    \foreach \n in {1,...,7}{
        \node[circle,fill=black] at ({\n*360/7}:2cm) (n\n) {};
        \node[anchor=north east] at ({\n*360/7}:2cm) {$\mv_{\n}$};
        \draw (center)--(n\n);
        \node at (0,-2*1.5) {}; 
        \node[anchor=south east] at (center) {$\mv_{c}$};
    }
\end{tikzpicture}
\caption{Star graph}\label{fig:stargraph}
\end{figure}
\end{definition}

\begin{example}[Equilateral Neumann star graph]\label{ex:eqst}
		Let $\mG$ be a star graph having $N$ edges, with central vertex $\mv_c$ and boundary vertices $\mv_1,\mv_2,\dots ,\mv_N$. Assume that $Q=\diag(q_{\mv})_{\mv\in\mV}$ is diagonal, $\ell_{\me}=\ell>0$, $p_{\me}=0$ and $c_{\me}=1$ for all $\me\in\mE$ (hence, $A$ is the Laplacian on $\mG$).  
        Then if $q_{\mv}= 0$ holds for more than one boundary vertex,  the transition semigroup $\mathcal{P}_t$ defined in \eqref{eq:transsgrnet} is not strong Feller at any time $T>0$.
\end{example}
 \begin{proof}
        Parameterize the edges on $\left[0,\ell\right]$ so that the $0$'s correspond to the boundary vertices.  By \cite[Ex.~2.1.12]{BeKu}  and \cite[Thm.~3.7.1]{BeKu} there is a sequence of distinct eigenvalues of $A$ with 
				\[-\mu_k=-\frac{\left(\frac12+k\right)^2\pi^2}{\ell^2},\quad k\in\nat,\] 
				and all eigenvalues have multiplicity $N-1$, see also \cite{Be85,BeLu05}. For each $k\in\nat$, there are $N-1$ linearly independent normalized eigenfunctions of the form
        \begin{equation}\label{eq:gkjs}
        \begin{split}
        &g_{k,1}(x)=\sqrt{\frac{1}{\ell}}\left(\cos(\sqrt{\mu_k}x),-\cos(\sqrt{\mu_k}x),0,0,\dots,0\right)\\
        &g_{k,2}(x)=\sqrt{\frac{1}{\ell}}\left(\cos(\sqrt{\mu_k}x),0,-\cos(\sqrt{\mu_k}x),0,0,\dots\right)\\
        &\vdots\\
        & g_{k,N-1}(x)=\sqrt{\frac{1}{\ell}}\left(\cos(\sqrt{\mu_k}x),0,0,\dots,-\cos(\sqrt{\mu_k}x)\right),  
        \end{split}
        \end{equation}
where the coordinates are the values of the functions on the $N$ edges. Hence, the eigenfunctions  become $0$ when evaluated at the central vertex $\mv_c$, while
\[g_{k,j}(\mv_1)=\sqrt{\frac{1}{\ell}},\;g_{k,j}(\mv_{j+1})=-\sqrt{\frac{1}{\ell}},\quad j=1,2,\dots ,N-1,\]
and
\[g_{k,j}(\mv_i)=0\quad i\geq 2,\, i\neq j+1,\quad j=1,2,\dots ,N-1.\] By the assumption on $Q$ and with an appropriate labelling of the vertices, that is, of the edges, there exists $j$ such that for each $k\in\nat$
    \[\left(Q^{\frac{1}{2}}Lg_{k,j}\right)_{\mv}=\sqrt{q_{\mv}}g_{k,j}(\mv)= 0\text{ for each }\mv\in\mV.\]
    Finally, recall that $B^*=-L$ by Proposition \ref{prop:Bstar}. Thus for all eigenfunctions $g_{k,j}$, $k\in\nat$ the right-hand-side, hence the left-hand-side of \eqref{eq:0contrprf1} would be zero for all $z_0\in\mcH$ which is impossible.
    Hence, by Proposition \ref{prop:0contrformula}, the problem \eqref{eq:netcontrol} is not null controllable in any time $T>0$, and equivalently, the transition semigroup $\mathcal{P}_t$ defined in \eqref{eq:transsgrnet} is not strong Feller at any time $T>0$. 
 \end{proof}

We also show an example for general star graphs with different edge lengths. Here we have to assume that the edge lengths for which no boundary noise is present, fulfill a certain rational dependence condition.

\begin{example}[General Neumann star graph]\label{cor:st}
Let $\mG$ be a star graph having $N$ edges, with central vertex $\mv_c$  and boundary vertices $\mv_1,\mv_2,\dots ,\mv_N$. Assume that $p_{\me}=0$, $c_e=1$, the edge lengths are $\ell_i>0$, $i=1,2,\dots,N$,  $Q=\diag(q_{\mv})_{\mv\in\mV}$ is diagonal and  $q_{\mv}= 0$ holds for at least two boundary vertices $\mv_1$ and $\mv_i$ such that 
\begin{equation}\label{eq:ell1elli}
    \frac{\ell_1}{\ell_i}=\frac{2n_1+1}{2n_i+1}\text{ for some }n_1,n_i\in\nat.
\end{equation} Then the transition semigroup $\mathcal{P}_t$ defined in \eqref{eq:transsgrnet} is not strong Feller at any time $T>0$.
 \end{example} 
 \begin{proof}
  We parametrize the edges on $\left[0,\ell_i\right]$,  $i=1,2,\dots,N$, so that the $0$'s correspond to the boundary vertices. We are going to define an eigenfunction supported on the edges $1$ and $i$ with 
 \begin{equation}\label{eq:fki}
        f_{k,i}(x)=\left(\cos(\sqrt{\mu_k}x),0,0,\dots,s_{k,i}\cos(\sqrt{\mu_k}x)),0,\dots ,0\right).      
        \end{equation}
Since noise is not present in $\mv_1$ and $\mv_i$, for $z_0=f_{k,i}$ \eqref{eq:0contrprf1} can not be satisfied. Thus, by Proposition \ref{prop:0contrformula} and the equivalence of the null controllability to the strong Feller property we obtain the statement.

By assumption \eqref{eq:ell1elli} there are integers $n_1$ and $n_i$ such that      \begin{equation}\label{eq:ell1perelli}
            \mu_k=\frac{(n_i+\frac{1}{2})^2\pi^2}{\ell_i^2}=\frac{(n_1+\frac{1}{2})^2\pi^2}{\ell_1^2}.
        \end{equation}
Let $f_{k,i}$ be the normalized eigenfunction as in \eqref{eq:fki} with this $\mu_k$ and with appropriate coefficients $r_{k,i}$, $s_{k,i}$ which we specify below.  From \eqref{eq:ell1perelli} follows that
\begin{equation}\label{eq:continmvc}
            \cos(\sqrt{\mu_k} \ell_1)=\cos(\sqrt{\mu_k} \ell_i)=0
        \end{equation}
thus the continuity condition in $\mv_c$ is satisfied for $f_{k,i}.$       
Let $s_{k,i}=-1$ if the parity of $n_1$ and $n_i$ is the same, and let $s_{k,i}=1$ otherwise. Then
\begin{equation}
            \sin((n_1+\frac{1}{2})\pi)+s_{k,i}\sin((n_i+\frac{1}{2})\pi)=0
        \end{equation}
holds, hence
\begin{equation}\sqrt{\mu_k}\sin(\sqrt{\mu_k}\ell_1)+\sqrt{\mu_k}s_{k,i}\sin(\sqrt{\mu_k}\ell_i)=0
        \end{equation}
and thus the Kirchhoff--Neumann condition in $\mv_c$ is satisfied. The Neumann-condition in the boundary vertices holds by the definition \eqref{eq:fki}.
Thus, the proof is completed.
\end{proof}

\begin{remark}
    The above example can be generalized by ''attaching'' arbitrary trees to at most $N-2$ edges of the star graph with arbitrary edges lengths, such that the lengths of the remaining $2$ edges satisfy \eqref{eq:ell1elli}. We can define an eigenfunction in the same way as in \eqref{eq:fki}, which is supported only on these 2 edges. When noise is not present in the corresponding 2 boundary vertices, the system can not be null controllable, hence $\mathcal{P}_t$ is not strong Feller at any time $T>0$.
\end{remark}

 The next example shows that the existence of a loop always destroys the strong Feller property.  
\begin{example}\label{ex:loop}
    Assume that $\mG$ contains a loop, that is a chain of vertices \[\mv, \mv_1,\mv_2, \dots,\mv_n, \mv\] connected by edges, with each of the intermediate vertices $\mv_1,\mv_2, \dots, \mv_n$ having degree 2. Assume that $p_{\me}=0$, $c_e=1$. Then the transition semigroup $\mathcal{P}_t$ defined in \eqref{eq:transsgrnet} is not strong Feller at any time $T>0$. 
    By \cite[Rem.~2.1]{BeLi}, we may assume that there is a looping edge in $\mG$, that is, an edge having the same origin and boundary vertex. Then \cite[Ex.~3.3]{BeLi} implies that there is an eigenfunction $f_k$ of $A$ which is supported exclusively on the loop and, in particular, its value equals to zero on all vertices of $\mG$. Hence, as $B^*=-L$ by Proposition \ref{prop:Bstar}, it follows from Proposition \ref{prop:0contrformula} that the problem \eqref{eq:netcontrol} is not null controllable in any time $T>0$, and equivalently, the transition semigroup $\mathcal{P}_t$ defined in \eqref{eq:transsgrnet} is not strong Feller at any time $T>0$.
\end{example}

	\section{Existence of an invariant measure}\label{sec:invarmeasure}
	In this section, we are interested in the existence of an invariant measure for the general boundary noise problem \eqref{eq:opAKmax1gen}, that is, for \eqref{eq:stocauchygen}. A probability measure $\mu$ on $({H},\mathcal{B}({H}))$ is said to be \emph{invariant} for \eqref{eq:stocauchygen} if 
	\begin{align*}
	\int_{{H}}\mathcal{P}_t\phi(x)\mu(dx)=\int_{H}\phi(x)\mu(dx), 
	\end{align*}
 for all $\phi\in C_b({H}),$ and $t>0$, where $\mathcal{P}_t$ is the transition semigroup defined in \eqref{eq:transsgr}.\\
 
	\begin{theorem}\label{inv_measure}
		Assume that Assumptions \ref{ass:noisecontrol} are satisfied, moreover, $B\in\mathcal{L}(U,H_{-\frac{1}{2}})$ holds. Then the system \eqref{eq:opAKmax1gen}, that is, \eqref{eq:stocauchygen} has an invariant measure if and only if $\ker A$ is empty or there exists an orthonormal basis $\Lambda:=(f_n)_{n\in \Gamma}$ of $\ker A$ such that
		\begin{equation}\label{eq:inv1}
			\|Q^{\frac{1}{2}}B^*f_n\|=0, \qquad \forall f_n\in\Lambda.	      
		\end{equation}
	\end{theorem}
	\begin{proof}
        By \cite[Thm.~11.17.(iii)]{DPZbook}, the existence of the invariant measure is equivalent to the fact that
        \begin{equation}\label{eq:invarmeasurequiv}
			\int_0^{+\infty}\left\|\mathbb{T}_{-1}(t)BQ^{\frac{1}{2}}\right\|^2_{\HS(U,H)}dt<+\infty,   
        \end{equation}
		where $\HS$ denotes the Hilbert-Schmidt norm between the appropriate spaces. 
  
        Let $T>0$ and let $(f_k)_{k\in\mathbb{N}}$ be an orthonormal basis of $H$ formed by the eigenfunctions of $A$ with corresponding eigenvalues $-\lambda_k$, see \eqref{eq:fk}. Then, we have 
		\begin{align}
			\int_0^T\left\|\mathbb{T}_{-1}(t)BQ^{\frac{1}{2}}\right\|^2_{\HS(U,H)}dt&=\int_0^T\left\|Q^{\frac{1}{2}}B^*\mathbb{T}(t)\right\|^2_{\HS(H,U)}dt\notag\\
			&=\sum_{k\in\mathbb{N}}\int_0^Te^{-2\lambda_k t}dt\cdot\left\|Q^{\frac{1}{2}}B^*f_k\right\|^2\label{eq:invproof}
		\end{align}
		Now, assume that for all orthonormal basis $\Lambda$ of $\ker A$ there exists $f_0\in\Lambda$ such that $\|Q^{\frac{1}{2}}B^*f_0\|>0$. It follows from \eqref{eq:invproof} that 
		\begin{align*}
			\int_0^T\left\|\mathbb{T}_{-1}(t)BQ^{\frac{1}{2}}\right\|^2_{\HS(U,H)}dt& \geq \int_0^Te^{-2\lambda_0 t}dt\cdot\left\|Q^{\frac{1}{2}}B^*f_0\right\|^2\\
			&\geq T\left\|Q^\frac{1}{2}B^*f_0\right\|^2.
		\end{align*}
		Hence 
		\[\int_0^{+\infty}\left\|\mathbb{T}_{-1}(t)BQ^{\frac{1}{2}}\right\|^2_{\HS(U,H)}dt=+\infty.\]
		Consequently, the system \eqref{eq:opAKmax1gen} does not have an invariant measure. 
  
        For the converse, we assume that there exists an orthonormal basis $\Lambda:=(f_n)_{n\in\Gamma}$ of $\ker A$ such that \eqref{eq:inv1} holds. By completing $\Lambda$ to an orthonormal basis $(f_k)_{k\in \mathbb{N}}$ of  $H$ formed by the eigenfunctions of $A$ with corresponding eigenvalues $-\lambda_k$, it follows from \eqref{eq:invproof} that 
		\begin{align*}
			\int_0^T\left\|\mathbb{T}_{-1}(t)BQ^{\frac{1}{2}}\right\|^2_{\HS(U,H)}dt&= \sum_{\lambda_j>0} \int_0^Te^{-2\lambda_jt}dt\cdot\|Q^{\frac{1}{2}}B^*f_j\|^2\\
			&=\sum_{\lambda_j>0} \frac{1-e^{-2\lambda_jT}}{2\lambda_j}\cdot\|Q^{\frac{1}{2}}B^*f_j\|^2\\
			&=\sum_{\lambda_j>0} \frac{1-\e^{-2\lambda_jT}}{2\lambda_j}\cdot\|Q^{\frac{1}{2}}\mathbb{D}_1^*(1-A)f_j\|^2\\
			&=\sum_{\lambda_j>0} \frac{1-e^{-2\lambda_jT}}{2\lambda_j}\cdot\|Q^{\frac{1}{2}}\mathbb{D}_1^*(1+\lambda_j)f_j\|^2\\
			&\leq \sum_{\lambda_j>0} \frac{1+\lambda_j}{2\lambda_j}\cdot\|Q^{\frac{1}{2}}(1+\lambda_j)^{\frac{1}{2}}\mathbb{D}_1^*f_j\|^2\\
			&\leq \sup_{\lambda_j>0} \frac{1+\lambda_j}{2\lambda_j} \sum_{\lambda_j>0}\|Q^{\frac{1}{2}}(1+\lambda_j)^{\frac{1}{2}}\mathbb{D}_1^*f_j\|^2\\
			& \leq\frac{1}{2}\left(\frac{1}{\mu_1}+1\right)\|(1-A)^{\frac{1}{2}}\mathbb{D}_1Q^{\frac{1}{2}}\|^2_{\HS(U,H)}\\
   &=\frac{1}{2}\left(\frac{1}{\mu_1}+1\right)\|(1-A)^{-\frac{1}{2}}BQ^{\frac{1}{2}}\|^2_{\HS(U,H)},
		\end{align*}
		where $\mu_1$ is defined in \eqref{eq:muks} and we have used $B\coloneqq (1-A_{-1})\Dir_1$. Moreover, it follows from the assumptions that \[\|(1-A)^{-\frac{1}{2}}BQ^{\frac{1}{2}}\|_{\HS(U,H)}<\infty.\] Hence, by taking $T\to +\infty$, we obtain 
		that \eqref{eq:invarmeasurequiv} holds, thus the invariant measure exists.
	\end{proof}

    The result \cite[Thm.~9.1.1.(iii)]{DPZbook96} states that if the transition semigroup is strong Feller, then the invariant measure exists if and only if the semigroup $\mathbb T$ is exponentially stable. In the following we show that for the quantum graph problem \eqref{eq:stochnetKnoise} the assumption being strong Feller can be omitted.\\

    To obtain the result, we observe that by Proposition \ref{prop:Bstar}, $B^*$ can be extended as a bounded operator $B^*\in\mathcal{L}(\mcH_{\frac{1}{2}},\mcY)$. Thus, we conclude that \[B\in\mathcal{L}(\mcY,\mcH_{-\frac{1}{2}}).\]
    \begin{corollary}\label{cor:inv} Assume that $c_{\me}=1$, $\me\in\mE$. Then,
    for the network boundary noise problem \eqref{eq:stochnetKnoise} the invariant measure exists if and only if the semigroup $\mathbb T$ generated by $A$ in \eqref{eq:amain} is exponentially stable. In particular, if an invariant measure for \eqref{eq:stochnetKnoise} exists, then it is unique.
    \end{corollary}
    \begin{proof}
        Propositions \ref{prop:assumptionfornet} and \ref{prop:Bstar} imply that for the problem \eqref{eq:stochnetKnoise} the assumptions of the above theorem are satisfied. By \cite[Thm.~1]{Ku19}, if $\ker A$ is not empty, it is a 1 dimensional subspace and the corresponding eigenvector can be chosen to be strictly positive. Hence, by Theorem \ref{inv_measure}, the invariant measure exists if and only if $\ker A=\emptyset$, thus the claim follows.
    \end{proof}

\begin{appendices}

	 \section{Regularity of the solution}\label{sec:regularity}

We again consider system \eqref{eq:stochnetKnoise} having mild solution  
\[w(t)=\mathbb{T}(t)z_0+\int_0^t\mathbb{T}_{-1}(t-s)Bd\beta(s)\notag:=\mathbb{T}(t)z_0+K(t), \]
with $B\coloneqq (1-A_{-1})\Dir_1$, see \eqref{eq:ZKt}. In this section we show that the  uniform boundedness assumption on the vertex values of eigenfunctions of Theorem 3.8 in \cite{KS23} is unnecessary to obtain the desired regularity of $K(t)$, that is, of $w(t)$.

	\begin{proposition}\label{prop:solH14}
		For $\alpha<\frac{1}{4}$, the stochastic convolution process $K(\cdot)$ defined in \eqref{eq:ZKt} has a continuous version in $\mcH_{\alpha}$. In particular, $w(\cdot)$ has an $\mathcal{H}_\alpha$ continuous version.
	\end{proposition}
	\begin{proof}
		Similarily to the argument in the proof of \cite[Thm.~3.8]{KS23}, by a straightforward modification of \cite[Thm.~2.3]{DPZ93} to include the covariance matrix $Q$, we have to show that if $\alpha<\frac{1}{4}$, then there exists $\gamma>0$ such that
		\begin{equation*}
			\int_0^{T}t^{-\gamma}\left\|\sgrT_{-1}(t)B Q^\frac12\right\|^2_{\HS(\mcY,\mcH_{\alpha})}\dt <+\infty.
		\end{equation*}
		By \cite[Thm.~4.1]{BKKS23}, 
	for $\alpha<\frac{1}{4}$ and $\ve>0$ small enough, $(1-A)^{\frac{1}{2}+\alpha+\ve}\Dir_1$ is a bounded operator from $\mathcal{Y}$ to $\mathcal{H}$. Hence, $B\in\mathcal{L}\left(\mathcal{Y},\mathcal{H}_{\alpha+\varepsilon-\frac{1}{2}}\right)$ holds for any $\alpha<\frac{1}{4}$ and $\varepsilon>0$ small enough. Thus,
	\begin{equation}
 \begin{split}
				&\int_0^{T}t^{-\gamma}\left\|\sgrT_{-1}(t)B Q^\frac12\right\|^2_{\HS(\mcY,\mcH_{\alpha})}\dt\\
                &=\int_0^{T}t^{-\gamma}\left\|(1-A)^{\frac{1}{2}-\varepsilon}\sgrT(t)(1-A)^{\alpha+\varepsilon-\frac{1}{2}}B Q^\frac12\right\|^2_{\HS(\mcY,\mcH)}\dt\\
				&\leq\int_0^{T}t^{-\gamma}\left\|(1-A)^{\frac{1}{2}-\varepsilon}\sgrT(t)\right\|^2\dt \cdot \left\|(1-A)^{\alpha+\varepsilon-\frac{1}{2}}B\right\|^2_{\HS(\mcY,\mcH)}\cdot\Tr(Q)\\
				&\leq c_T\cdot \int_0^{T}t^{-\gamma}\cdot \frac{1}{t^{1-2\ve}}\dt, 	
 \end{split}
 \end{equation}
where, in the last inequality we used the analiticity of the semigroup $\mathbb{T}$ and the fact that the Hilbert-Schmidt norm of $(1-A)^{\alpha+\varepsilon-\frac{1}{2}}B$  is finite since $\mcY$ is finite dimensional. The last integral is finite if $\gamma<2\ve$, which implies the statement.
	\end{proof}

	 \section{ST active set of a tree}\label{sec:STtree}

In this section we show how \cite[Thm.~3.14]{AvdZhao22} can be used in the proof of Theorem \ref{thm:main}. To keep the paper self-contained, we cite all necessary definitions and results from \cite{AvdZhao21, AvdZhao22} and prove in Proposition \ref{prop:STtree} how their apply for the situation in Theorem \ref{thm:main}.\\

We denote by $\Vec{\mG}=(\mV,\Vec{\mE})$ a directed graph obtained by orienting the edges of $\mG=(\mV,\mE)$. We call a directed graph \emph{acyclic}, if there is no directed circle (directed path having the same vertex as starting and finishing point, see Definition \ref{defi:treeGamma}) in it. By \cite[Lem.~2.2]{AvdZhao22}, there is always a directed acyclic graph -- abbreviated as DAG --, $\Vec{\mG}=(\mV,\Vec{\mE})$ based on $\mG$.

\begin{definition}{\cite[Def.~3.1]{AvdZhao22}}\label{defi:TFPU}
    Let $\Vec{\mG}=(\mV,\Vec{\mE})$ be a DAG (directed acyclic graph) based on $\mG$. Let $U$ be a union of directed paths. We say $U$ is a \emph{tangle-free (TF) path union} of $\mG$ if $U$ satisfies the conditions:
    \begin{enumerate}
        \item The direction of all edges are the same as the direction of all paths they are on.
        \item All paths $P\in U$ are disjoint except for the starting and finishing vertices.
        \item If a finishing vertex $\mv$ of a path is the starting vertex of another path, there must be an incoming edge of $\mv$ that is not a finishing edge, and an outgoing edge of $\mv$ that is not a starting edge.
        \item $\Vec{\mG}=\bigcup_{P\in U}P$.
    \end{enumerate}
\end{definition}

As before, $\mE_{\mv}$ denotes the set of all edges which are incident to $\mv$. In $\Vec{\mG}$, we denote by $\mE^+_{\mv}$ and $\mE^-_{\mv}$ the set of edges which are outgoing and incoming edges, respectively, for $\mv$. A \emph{source} is a vertex without incoming edges; a \emph{sink} is a vertex without outgoing edges. We denote the sets of sources and sinks by $\Vec{\mG}^+$ and $\Vec{\mG}^-$, respectively.

The following concept is crucial in the proof of Theorem \ref{thm:main}.

\begin{definition}{\cite[Def.~3.2]{AvdZhao22}} \label{defi:STactive}
Let $\Vec{\mG}$ be a DAG based on $\mG$.
    We say that $I^*$ is a \emph{single-track (ST) active set of vertices} of $\Vec{\mG}$ if
    \[I^*=\{\mv: \mv\in \Vec{\mG}^+\},\]
    $J^*$ is a \emph{single-track (ST) active set of edges} of $\Vec{\mG}$ if
    \begin{equation}\label{eq:Jstar}
        J^*=\bigcup_{\mv\in\mV} \mE^*_{\mv},
    \end{equation}
    with 
    \[\mE^*_{\mv}=\text{all but one elements of }\mE^+_{\mv},\; \mv\in\mV.\]
    The set $\{I^*,J^*\}$ is then called a \emph{ST active set}.
\end{definition}
Observe that $I^*$ is uniquely determined by $\Vec{\mG}^+$. However, ST active set of edges can be defined in several ways. By the following result, there is a one-to-one correspondence between the ST active set of edges and the TF path unions  of $\mG$.

\begin{lemma}{\cite[Lem.~3.3]{AvdZhao22}}\label{lem:TFPUST}
 Let  $\Vec{\mG}$ be a DAG and $I^*$ its ST active set of vertices. Then each TF path union $U$ corresponds to a ST active set of edges $J^*$ such that for every path in $U$, the index of its starting edge is in
 \[\overline{J}^*=\left(\bigcup_{\mv\in I^*}\mE_{\mv}\right)\cup J^*.\]
 Conversely, from every ST active set of edge indices $J^*$ one can construct a TF path union $U$, such that for every $\me\in \overline{J}^*$, $\me$ is the starting edge of a path in $U$.  
\end{lemma}

Let $\mG$ be a tree. Our upcoming result is based on the following lemma. We recall that we denote the set of boundary vertices -- that is, vertices of degree $1$ -- by $\Gamma$, and $P(\mv,\mv')$ denotes the uniquely determined path in $\mG$ with starting point $\mv$ and finishing point $\mv'$, see Definition \ref{defi:treeGamma}.

\begin{lemma}{\cite[Lem.~1]{AvdZhao21}}\label{lem:treepathU}
    Let $\mG$ be a tree graph  and $I\subset \Gamma$ be a subset of its boundary vertices. Then $\mG$ has a union representation
    \begin{equation}\label{eq:pathtree}
        \mG=\bigcup_{\mv\in I} P(\mv,\mv'),
    \end{equation}    
    of disjoint (except the endpoints) paths if and only if 
    \begin{equation}\label{eq:IGammagamma1}
     I=\Gamma\text{ or }I=\Gamma\setminus\{\gamma_1\}  
    \end{equation}
    for an arbitrary $\gamma_1\in\Gamma$.
\end{lemma}

Notice that a path union representation \eqref{eq:pathtree} in the tree graph $\mG$ uniquely determines an orientation of the edges that is a directed graph, $\Vec{\mG}$, such that each edge is oriented according to the orinetation of the (unique) path it is contained in.

\begin{proposition}\label{prop:STtree}
   For a tree graph $\mG$ and the vertex set $I=\Gamma$ or $I=\Gamma\setminus\{\gamma_1\}$ (with $\gamma_1\in\Gamma$ arbitrary), consider a path union representation \eqref{eq:pathtree} from Lemma \ref{lem:treepathU} and the corresponding directed graph $\Vec{\mG}$. Then  \eqref{eq:pathtree} is a TF path union of $\mG$ and the corresponding ST active set given by Lemma \ref{lem:TFPUST} is
   \[I^*= I,\quad J^*=\emptyset.\]
   Furthermore, $I^*\subset \Gamma$, $J^*=\emptyset$ is an ST active set if and only if $I^*=\Gamma$ or $I^*:=\Gamma\setminus\{\gamma_1\}$ for some $\gamma_1\in\Gamma$.
\end{proposition}
\begin{proof}
 Since $\mG$ is a tree, $\Vec{\mG}$ is a DAG based on $\mG$. It is straightforward that (1), (2) and (4) of Definition \ref{defi:TFPU} are satisfied for the path union \eqref{eq:pathtree}. Since all paths have starting points in boundary vertices, no interior vertex -- that is, a vertex of degree at least $2$ -- is a starting point of a path, hence (3) of Definition \ref{defi:TFPU} is also satisfied. This fact also implies, that each interior vertex has (at least one) incoming edge, and in the case $I=\Gamma\setminus\{\gamma_1\}$, $\gamma_1$ is a sink. Thus, the sources of $\Vec{G}$ are exactly the elements of $I$, that is, $I^*=I$.

Otherwise, since the paths are disjoint, and no interior vertex is a starting point for a path, all interior vertices as well as the elements of $I$ have exactly one outgoing edge. Thus, by \eqref{eq:Jstar}, $J^*=\emptyset$.

For the last statement observe that by Lemma \ref{lem:TFPUST}, there is a one-to-one correspondence between the ST active sets and the TF path unions of $\mG$ such that the starting deges of the paths are exactly the elements of $\overline{J}^*$. Since $I^*\subset \Gamma$, $J^*=\emptyset$, we have
\[\overline{J}^*=\bigcup_{\mv\in I^*}\{\me_\mv\},\]
where $\me_\mv$ is the unique edge incident to $\mv\in I^*$. Hence, the paths in the TF path union corresponding to $\{I^*,J^*\}$ have starting points only in $I^*$. By Lemma \ref{lem:treepathU}, this can happen if and only $I^*$ has the given form. 
\end{proof}

 \end{appendices}

\section{Acknowledgements}

This article is based upon work from COST Action 18232 MAT-DYN-NET, supported by COST (European Cooperation in Science and Technology), www.cost.eu. 
M. Fkirine is supported by the Research Council of Finland grant 349002.

M.~Kov\'acs acknowledges the support of the Hungarian National Research, Development and Innovation Office (NKFIH)
through Grant no. TKP2021-NVA-02 and Grant no. K-145934.

E.~Sikolya was supported by the OTKA grant no. 135241.

\section*{Declarations}

The authors declare that there are no relevant financial or non-financial competing interests to report.


\begin{thebibliography}{99}



\bibitem{Avd1}
\newblock S.~Avdonin,
\newblock Control problems on quantum graphs,
\newblock in \emph{Analysis on graphs and its applications}, vol.~77 of Proc.
  Sympos. Pure Math.,
\newblock Amer. Math. Soc., Providence, RI, 2008,
\newblock 507--521,
\url{https://doi.org/10.1090/pspum/077/2459889}.

\bibitem{AvdIvbook}
\newblock S.~Avdonin and S.~Ivanov, 
\newblock \emph{Families of Exponentials: The Method of Moments in Controllability Problems for Distributed Parameter Systems},
\newblock Cambridge University Press, 1995.
\newblock ISBN: 0521452430.


\bibitem{AvdMikh08}
\newblock S.~Avdonin and V.~Mikhaylov,
\newblock Controllability of partial differential equations on graphs,
\newblock \emph{Appl. Math. (Warsaw)}, \textbf{35} (2008), 379--393,
\url{https://doi.org/10.4064/am35-4-1}.

\bibitem{AvdZhao21}
\newblock S.~Avdonin and Y.~Zhao,
\newblock Exact controllability of the 1-D wave equation on finite metric tree graphs,
\newblock \emph{Appl. Math. Optim.}, \textbf{83} (2021), 2303--2326,
\url{https://doi.org/10.1007/s00245-019-09629-3}.


\bibitem{AvdZhao22}
\newblock S.~Avdonin and Y.~Zhao,
\newblock Exact controllability of the wave equation on graphs,
\newblock \emph{Appl. Math. Optim.}, \textbf{85} (2022), Paper No. 1, 44,
\url{https://doi.org/10.1007/s00245-022-09869-w}.

\bibitem{Barcena}
\newblock J. A. Bárcena-Petisco, M. Cavalcante, G. M. Coclite, N. De Nitti and E. Zuazua, 
\newblock Control of hyperbolic and parabolic equations on networks and singular limits, 
\newblock \emph{Math.~Control Relat.~F.}, \textbf{15}(1) (2025), 348--389, 
\url{https://www.aimsciences.org/article/doi/10.3934/mcrf.2024015}.

\bibitem{Bel}
\newblock M.~I. Belishev,
\newblock On the boundary controllability of a dynamical system described by
  the wave equation on a class of graphs (on trees),
\newblock \emph{Zap. Nauchn. Sem. S.-Peterburg. Otdel. Mat. Inst. Steklov.
  (POMI)}, \textbf{308} (2004), 23--47, 252,
\url{https://doi.org/10.1007/s10958-005-0471-x}.

\bibitem{Be17}
\newblock G.~Berkolaiko,
\newblock An elementary introduction to quantum graphs,
\newblock in \emph{Geometric and computational spectral theory}, vol. 700 of
  Contemp. Math.,
\newblock Amer. Math. Soc., Providence, RI, 2017,
\newblock 41--72,
\url{https://doi.org/10.1090/conm/700/14182}.

\bibitem{BeKu}
\newblock G.~Berkolaiko and P.~Kuchment,
\newblock \emph{Introduction to quantum graphs}, vol. 186 of Mathematical
  Surveys and Monographs,
\newblock American Mathematical Society, Providence, RI, 2013,
\url{https://doi.org/10.1090/surv/186}.

\bibitem{BeLi}
\newblock G.~Berkolaiko and W.~Liu,
\newblock Simplicity of eigenvalues and non-vanishing of eigenfunctions of a
  quantum graph,
\newblock \emph{J. Math. Anal. Appl.}, \textbf{445} (2017), 803--818,
\url{https://doi.org/10.1016/j.jmaa.2016.07.026}
\bibitem{BKKS23}
\newblock D.~Bolin,~M.~Kov\'acs, V. Kumar and A.~B. Simas,
\newblock Regularity and numerical approximation of fractional elliptic
  differential equations on compact metric graphs,
\newblock \emph{Math. Comp.}, Math. Comp. \textbf{93} (2024), 2439-2472,
\url{https://doi.org/10.1090/mcom/3929}.

\bibitem{BKS18}
\newblock J.~Boman, P.~Kurasov and R.~Suhr,
\newblock Schrödinger Operators on Graphs and Geometry II. Spectral Estimates for $L_1$-potentials and an Ambartsumian Theorem,
\newblock \emph{Integr. Equ. Oper. Theory}, \textbf{90:40} (2018),
\url{https://doi.org/10.1007/s00020-018-2467-1}


\bibitem{DPZbook96}
\newblock G.~Da~Prato and J.~Zabczyk,
\newblock \emph{Ergodicity for infinite-dimensional systems}, vol. 229 of
  London Mathematical Society Lecture Note Series,
\newblock Cambridge University Press, Cambridge, 1996,
\url{https://doi.org/10.1017/CBO9780511662829}.

\bibitem{DPZ93}
\newblock G.~Da~Prato and J.~Zabczyk,
\newblock Evolution equations with white-noise boundary conditions,
\newblock \emph{Stochastics Stochastics Rep.}, \textbf{42} (1993), 167--182,
\url{https://doi.org/10.1080/17442509308833817}.

\bibitem{DPZbook}
\newblock G.~Da~Prato and J.~Zabczyk,
\newblock \emph{Stochastic equations in infinite dimensions}, vol. 152 of
  Encyclopedia of Mathematics and its Applications,
\newblock 2nd edition,
\newblock Cambridge University Press, Cambridge, 2014,
\url{https://doi.org/10.1017/CBO9781107295513}.

\bibitem{Zuazuabook}
\newblock R.~D\'{a}ger and E.~Zuazua,
\newblock \emph{Wave propagation, observation and control in {$1\text{-}d$}
  flexible multi-structures}, vol.~50 of Math\'{e}matiques \& Applications
  (Berlin) [Mathematics \& Applications],
\newblock Springer-Verlag, Berlin, 2006,
\url{https://doi.org/10.1007/3-540-37726-3}.



\bibitem{fatt71}
\newblock H.~O. Fattorini and D.~L. Russell,
\newblock Exact controllability theorems for linear parabolic equations in one
  space dimension,
\newblock \emph{Arch. Rational Mech. Anal.}, \textbf{43} (1971), 272--292,
\url{https://doi.org/10.1007/BF00250466}.

\bibitem{FHR23}
\newblock M.~Fkirine, S.~Hadd and A.~Rhandi,
\newblock On evolution equations with white-noise boundary conditions,
\newblock \emph{J. Math. Anal. Appl.}, \textbf{535} (2024), Paper no. 128087,
\url{https://doi.org/10.1016/j.jmaa.2024.128087}.

\bibitem{FHR25}
\newblock M.~Fkirine, S.~Hadd and A.~Rhandi,
\newblock Impact of mixed boundary conditions on stochastic equations with noise at the boundary, 
\newblock \emph{Nonlinear Differ.~Equ.~ Appl.} \textbf{32}, 30 (2025).
\url{https://doi.org/10.1007/s00030-025-01032-y}

\bibitem{Gr87}
\newblock G.~Greiner,
\newblock Perturbing the boundary conditions of a generator,
\newblock \emph{Houston J. Math.}, \textbf{13} (1987), 213--229.

\bibitem{KMW03}
\newblock J.~P. Keating, J.~Marklof and B.~Winn,
\newblock Value distribution of the eigenfunctions and spectral determinants of
  quantum star graphs,
\newblock \emph{Comm. Math. Phys.}, \textbf{241} (2003), 421--452,
\url{https://doi.org/10.1007/s00220-003-0941-2}.

\bibitem{KS23}
\newblock M.~Kov\'{a}cs and E.~Sikolya,
\newblock On the parabolic {C}auchy problem for quantum graphs with vertex
  noise,
\newblock \emph{Electron. J. Probab.}, \textbf{28} (2023), Paper No. 74, 20,
\url{https://doi.org/10.1214/23-ejp962}.

\bibitem{Ku19}
\newblock P.~Kurasov,
\newblock On the ground state for quantum graphs,
\newblock \emph{Lett. Math. Phys.}, \textbf{109} (2019), 2491--2512,
\url{https://doi.org/10.1007/s11005-019-01192-w}.

\bibitem{LLS94}
\newblock J.~E. Lagnese, G.~Leugering and E.~J. P.~G. Schmidt,
\newblock \emph{Modeling, analysis and control of dynamic elastic multi-link
  structures},
\newblock Systems \& Control: Foundations \& Applications, Birkh\"{a}user
  Boston, Inc., Boston, MA, 1994,
\url{https://doi.org/10.1007/978-1-4612-0273-8}.

\bibitem{Mu14}
\newblock D.~Mugnolo,
\newblock \emph{Semigroup methods for evolution equations on networks},
\newblock Understanding Complex Systems, Springer, Cham, 2014,
\url{https://doi.org/10.1007/978-3-319-04621-1}.

\bibitem{Russel78}
\newblock D.~L. Russell,
\newblock Controllability and stabilizability theory for linear partial
  differential equations: recent progress and open questions,
\newblock \emph{SIAM Rev.}, \textbf{20} (1978), 639--739,
\url{https://doi.org/10.1137/1020095}.

\bibitem{TuWe09}
\newblock M.~Tucsnak and G.~Weiss,
\newblock \emph{Observation and control for operator semigroups},
\newblock Birkh\"{a}user Advanced Texts: Basler Lehrb\"{u}cher. [Birkh\"{a}user
  Advanced Texts: Basel Textbooks], Birkh\"{a}user Verlag, Basel, 2009,
\url{https://doi.org/10.1007/978-3-7643-8994-9}.

\bibitem{vNeervenbook}
\newblock J.~van Neerven,
\newblock \emph{The adjoint of a semigroup of linear operators}, vol. 1529 of
  Lecture Notes in Mathematics,
\newblock Springer-Verlag, Berlin, 1992,
\url{https://doi.org/10.1007/BFb0085008}.

\bibitem{Be85}
\newblock J.~von Below,
\newblock A characteristic equation associated to an eigenvalue problem on
  {$c^2$}-networks,
\newblock \emph{Linear Algebra Appl.}, \textbf{71} (1985), 309--325,
\url{https://doi.org/10.1016/0024-3795(85)90258-7}.

\bibitem{BeLu05}
\newblock J.~von Below and J.~A. Lubary,
\newblock The eigenvalues of the {L}aplacian on locally finite networks,
\newblock \emph{Results Math.}, \textbf{47} (2005), 199--225,
\url{https://doi.org/10.1007/BF03323026}.

\bibitem{Weiss89}
\newblock G.~Weiss,
\newblock Admissibility of unbounded control operators,
\newblock \emph{SIAM J. Control Optim.}, \textbf{27} (1989), 527--545,
\url{https://doi.org/10.1137/0327028}.

\bibitem{weiss1989admissible}
\newblock G.~Weiss,
\newblock Admissible observation operators for linear semigroups,
\newblock \emph{Isr.~J.~Math.}, \textbf{65} (1989), 17--43,
\url{https://doi.org/10.1007/BF02788172}.

\end{thebibliography}
\end{document}